\documentclass[reqno,a4paper]{amsart}

\usepackage[latin1]{inputenc}
\usepackage[english]{babel}
\usepackage{eucal,amsfonts,amssymb,amsmath,amsthm,epsfig,mathrsfs}
\usepackage{cancel,soul}
\usepackage{color}
\usepackage{amscd,amsxtra}
\usepackage{enumerate}
\usepackage{latexsym}
\usepackage{bm}

\allowdisplaybreaks

\newcounter{ipotesi}

\Alph{ipotesi}

\newtheorem{thm}{Theorem}[section]
\newtheorem{hyp}[thm]{Hypotheses}{\rm}
\newtheorem{hyp0}[thm]{Hypothesis}{\rm}
\newtheorem{lemm}[thm]{Lemma}
\newtheorem{cor}[thm]{Corollary}
\newtheorem{prop}[thm]{Proposition}

\newtheorem{rmk}[thm]{Remark}{\rm}

\newcounter{parentenv}

\newcommand{\R}{{\mathbb R}}
\newcommand{\CC}{{\mathbb C}}
\newcommand{\E}{{\mathbb E}}
\newcommand{\N}{{\mathbb N}}

\newcommand{\Rn}{\mathbb R^n}

\newcommand{\Om}{{\Omega}}

\newcommand{\supp}{{\rm{supp}}\,}

\newcommand{\one}{\mbox{$1\!\!\!\;\mathrm{l}$}}

\newcommand{\ve}{\varepsilon}
\newcommand{\eps}{\varepsilon}
\newcommand{\ra}{\rightarrow}

\renewcommand{\tilde}[1]{\widetilde{#1}}

\newcommand{\set}[1]{{\left\{#1\right\}}}
\newcommand{\pa}[1]{{\left(#1\right)}}
\newcommand{\sq}[1]{{\left[#1\right]}}
\newcommand{\gen}[1]{{\left\langle #1\right\rangle}}
\newcommand{\abs}[1]{{\left|#1\right|}}
\newcommand{\norm}[1]{{\left\|#1\right\|}}

\newcommand{\eqsys}[1]{{\left\{\begin{array}{ll}#1\end{array}\right.}}
\newcommand{\tc}{\, \middle |\,}

\begin{document}

\frenchspacing

\title[Gradient estimates on infinite dimensional convex domains]{Gradient estimates for
perturbed Ornstein-Uhlenbeck semigroups on infinite dimensional convex domains}

\author[L. Angiuli]{{L. Angiuli}$^*$}\thanks{$^*$Corresponding author}

\author[S. Ferrari]{{S. Ferrari}}

\author[D. Pallara]{{D. Pallara}}

\address[L. Angiuli, S. Ferrari and D. Pallara]{Dipartimento di Matematica e Fisica ``Ennio de Giorgi'', Universit\`a del Salento, Via per Arnesano snc, 73100 Lecce, Italy.}
\email{\textcolor[rgb]{0.00,0.00,0.84}{luciana.angiuli@unisalento.it}}
\email{\textcolor[rgb]{0.00,0.00,0.84}{simone.ferrari@unisalento.it}}
\email{\textcolor[rgb]{0.00,0.00,0.84}{diego.pallara@unisalento.it}}


\begin{abstract}
Let $X$ be a separable Hilbert space endowed with a non-degenerate centred Gaussian measure $\gamma$ and let $\lambda_1$ be the maximum eigenvalue of the covariance operator associated with $\gamma$. The associated Cameron--Martin space is denoted by $H$. For a sufficiently regular convex function $U:X\ra\R$ and a convex set $\Omega\subseteq X$, we set $\nu:=e^{-U}\gamma$ and we consider the semigroup
$(T_\Omega(t))_{t\geq 0}$ generated by the self-adjoint operator defined via the quadratic form
\[
(\varphi,\psi)\mapsto \int_\Omega\gen{D_H\varphi,D_H\psi}_Hd\nu,
\]
where $\varphi,\psi$ belong to $D^{1,2}(\Omega,\nu)$, the Sobolev space defined as the domain of the closure in $L^2(\Omega,\nu)$ of $D_H$, the gradient operator along  the directions of $H$.

A suitable approximation procedure allows us to prove some pointwise gradient estimates for $(T_\Om(t))_{t\ge 0}$. In particular, we show that
\[
|D_H T_\Om(t)f|_H^p\le e^{- p \lambda_1^{-1} t}(T_\Om(t)|D_H f|^p_H),
\qquad\, t>0,\ \nu\text{\rm -a.e. in }\Om,
\]
for any $p\in [1,+\infty)$ and $f\in D^{1,p}(\Om ,\nu)$. We deduce some relevant consequences of the previous estimate, such as the logarithmic Sobolev inequality and the Poincar\'e inequality in $\Om$ for the measure $\nu$ and some improving summability properties for $(T_\Omega(t))_{t\geq 0}$. In addition we prove that if $f$ belongs to $L^p(\Omega,\nu)$ for some $p\in(1,\infty)$, then
\[|D_H T_\Omega(t)f|^p_H \leq K_p t^{-\frac{p}{2}} T_\Omega(t)|f|^p,\qquad \, t>0,\ \nu\text{-a.e. in }\Omega,\]
where $K_p$ is a positive constant depending only on $p$.
Finally we investigate on the asymptotic behaviour of the semigroup $(T_\Om(t))_{t\geq 0}$ as $t$ goes to infinity.
\end{abstract}

\date{\today}

\maketitle

\section*{Introduction}

This paper is a contribution to the study of infinite dimensional elliptic and parabolic partial differential equations.
The basic data are an abstract Wiener space $(X,H,\gamma)$ and a quadratic form
which defines a self-adjoint operator. This is a recent field of research, that finds its main motivation in stochastic analysis and its different applications to mathematical finance, statistical mechanics, hydrodynamics and quantum mechanics.
The simplest (still, quite challenging) case is that of a Hilbert space $X$ endowed with a Gaussian measure $\gamma$ and the Dirichlet form
\[
(\varphi,\psi) \mapsto \int_X\langle D_H\varphi,D_H\psi\rangle_H\, d\gamma ,
\]
that defines an Ornstein-Uhlenbeck operator $L$ which in turn generates the associated Ornstein-Uhlenbeck semigroup. Here $D_H$ denotes the gradient along the directions of Cameron-Martin space $H$. Much has been done on this subject, see \cite{Bog,DaP02,Hin03,Hin11,MvN07,MvN11}, relying on the available explicit Mehler's formula for the semigroup. In this case, the related stochastic differential equation is the Langevin one, i.e.,
\[
 dX(t) = - X(t)\, dt + dW^H(t),
\]
where $W^H(t)$ is a cylindrical Brownian motion.
It is natural to look for generalisations of the available results, going in two directions: one is that of replacing $\gamma$ with a more general measure, the other is that of considering integration on a domain $\Omega\subseteq X$. One of the main properties of Gaussian measures is that they factor according to the orthogonal decompositions of $H$, and this allows to get explicit formulas when integrating on the whole space $X$ and to perform finite dimensional approximations with increasing sequences of subspaces. Moreover, integrating on a domain requires to deal with boundary conditions (or suitable classes of test functions) that have to be assigned in order to correctly define an operator and the generated semigroup. Introducing a different measure makes the finite dimensional approximation much more delicate and prevents to get explicit formulas even if the problem is studied in the whole space. Restricting to a domain, beside involving boundary conditions that have to be understood, makes still more difficult the infinite dimensional approximation, and in fact, to the best of our knowledge, the only case treated in the literature is that of convex domains, see \cite{ACF17, BDaPT1, BDaPT2, BDaPT3, Cap15, CF18, DL10, DL15, LMP}.

In this paper we consider a log-concave weighted Gaussian measure $\nu = e^{-U}\gamma$ on a separable Hilbert space $X$. Here $\gamma={\mathcal N}(0,Q_\infty)$ is the Gaussian measure with zero mean and covariance operator $Q_\infty:=-QA^{-1}$ where
$Q$ is a self-adjoint bounded non-negative and non-degenerate operator on $X$, $A:D(A)\subseteq X\to X$ is a self-adjoint operator such that $\langle Ax,x\rangle \leq -\omega |x|^2$ ($\omega > 0$) and $Q_\infty$ is a trace-class operator with non-negative eigenvalues $(\lambda_i)_{i\in\N}$. The function $U:X\to\R$ is convex and sufficiently regular (precise hypotheses are stated in Section \ref{prelim-sect}). We consider the quadratic form
\begin{equation}\label{Dirnu}
 {\mathcal D}_\Omega(\varphi,\psi) = \int_{\Omega}\langle D_H\varphi,D_H\psi\rangle_H\, d\nu ,
\end{equation}
which gives rise to the Kolmogorov operator (formally defined in a variational way through ${\mathcal D}_\Omega$)
\[
 L={\rm Tr}(D^2_H )-\sum_{i=1}^{+\infty}\lambda_i^{-1}x_i D_i
 -\langle D_HU,D_H\rangle_H
\]
and to the stochastic differential equation
\begin{equation}\label{star}
 dX(t) = - (X(t)+DU(X(t)))\, dt + Q_\infty^{1/2}\, dW^H(t) + \text{boundary terms},
\end{equation}
(we do not enter into the details of boundary terms because we shall not come back to the stochastic side, see \cite{BDaPT1,BDaPT2} for a precise formulation of the equation \eqref{star}).
The domain we assign to the quadratic form corresponds heuristically to Neumann boundary conditions for $L$ on $\partial\Omega$, and $L$ generates a strongly continuous semigroup $(T_\Omega(t))_{t\ge 0}$ (simply denoted by $T_\Om(t)$) in $L^p(\Omega,\nu)$ for $1\leq p<\infty$.
In order to study this semigroup, we proceed with a double approximation. We approximate $U$ via Moreau-Yosida type operators and penalise the characteristic function of $\Omega$ in order to state the problem in the whole space, eventually getting the restriction to $\Omega$ when the penalisation converges to
$\chi_\Omega$. It is here that the convexity assumption on $\Omega$ is essential. Indeed, in infinite dimension there is no available procedure to mimic the standard domain decomposition and partition of unity arguments which are classical in finite dimension. Once the (approximate) problem has been formulated in the whole space, we perform a finite dimensional approximation which provides a quite regular family of semigroups converging to $T_\Om(t)f$ in a suitable sense and to which the results of the finite dimensional case can be applied.

As we don't know any smoothing property of $T_\Omega(t)$ (it is not even known whether $T_\Omega(T)$ maps $C_b(\Omega)$ in $C_b(\Omega)$)), we exploit the smoothing properties of the approximating semigroups.
Indeed, the smoothness of the approximants is the crucial tool for many computations in this paper. Among the most relevant results that follow, there is the pointwise gradient estimate
\begin{equation}\label{g-e-intro}
|D_H T_\Om(t)f|_H^p\le e^{- p \lambda_1^{-1} t}(T_\Om(t)|D_H f|^p_H),\qquad\, t>0,\ \nu\text{\rm -a.e. in }\Om,
\end{equation}
which holds true for any $p\in [1,+\infty)$ and $f$ smooth enough, $\lambda_1$ being the maximum eigenvalue of the covariance operator $Q_\infty$.
Besides its own interest, estimate \eqref{g-e-intro} represents the key tool in the investigation of many qualitative properties of $T_{\Om}(t)$ and the related invariant measure $\nu$.
In the finite dimensional case, gradient estimates similar to \eqref{g-e-intro} are usually obtained by using the Bernstein method, which relies upon a variant of the classical maximum principle (see \cite{Lor17} and the reference therein) that does not have a counterpart in the infinite dimensional case, or by using stochastic techniques, such as the Bismut--Elworthy--Li formula (see \cite{Cer01,DaP02} and reference therein) and coupling methods (see for example \cite{Cra91, Cra92, Wan97}).
On the other hand, in infinite dimensional Wiener spaces some partial results are also available. In the case of a Gaussian measure $\gamma$ and $\Omega=X$, the classical Mehler's representation formula
\[
T(t)f(x)=\int_X f\pa{e^{-t}x+\sqrt{1-e^{-2t}}y}d\gamma(y),
\]
gives $D_H T(t)f=e^{-t}T(t)(D_H f)$, where the equality has to be meant componentwise, (see \cite[Proposition 1.5.6]{Bog}).
Again for the Gaussian measure $\gamma$ on a convex subset $\Omega$, in \cite[Theorem 3.1]{Cap15} it is proved that $|D_H T(t)f|_H\leq e^{-t}T(t)|D_H f|_H$ for any smooth function $f$. In this case, the idea consists in approximating the parabolic problem with a sequence of finite dimensional parabolic problems and using the factorisation of the Gaussian measure. Clearly, this approach does not work in our case since our measure in general does not decompose as a product of measures on orthogonal subspaces.
Finally, the case of a weighted Gaussian measure is also considered in \cite{DaP02} where a version of \eqref{g-e-intro} is proved when $\Omega=X$ and the $H$-derivative  is replaced by the Fr\'echet one.  We point out that, in this latter case, the proof of the gradient estimate is based on purely stochastic techniques.

Hence, taking account of the existing literature, estimate \eqref{g-e-intro} represents a generalisation of all the above results and the purely analytical proof we proposed, inspired by an idea due to Bakry and \'Emery (see \cite{BE85} and \cite{Sav14}), is a novelty in the proofs of gradient estimates.

As announced, the pointwise gradient estimate \eqref{g-e-intro} has several interesting consequences.
First of all it yields that the semigroup $T_\Om(t)$ is smoothing, in the sense that it is bounded from $L^p(\Om,\nu)$ into $D^{1,p}(\Om,\nu)$, for any $p\in(1,\infty)$ and $t>0$ as the estimate
\begin{align*}
\norm{D_H T_\Omega(t)f}_{L^p(\Omega,\nu;H)}
&\leq C_p t^{-\frac{1}{2}} \norm{f}_{L^{p}(\Omega,\nu)},
\end{align*}
reveals. Due to the fact that the Sobolev embedding theorems fail to hold when we replace the Lebesgue measure with another general measure (as the Gaussian one), despite $T_\Om(t)$ maps $L^p(\Om,\nu)$ into $D^{1,p}(\Om ,\nu)$, a natural basic question is whether the semigroup $T_\Om(t)$ is hypercontractive, i.e., if, given any $f\in L^q(\Om, \nu)$, $q\in[1,\infty)$, the function $T_\Om(t)f$ belongs to $L^p(\Om ,\nu)$ for some $p>q$. To give a positive answer, the starting point is the proof of a logarithmic Sobolev inequality for the measure $\nu$ which, as in the case of Gaussian measures, implies that the semigroup $T_\Om(t)$ is hypercontractive in the $L^p$-spaces related to the measure $\nu$.  We also show a Poincar\'e inequality in $L^p(\Om ,\nu)$ for $p\in[2,\infty)$ that together with the hypercontractivity estimate $\|T_\Om(t)f\|_{L^p(\Om, \nu)}\le c_{p,q,\Om}\|f\|_{L^q(\Om,\nu)}$ which holds for any $t>0$, $f\in L^q(\Om,\nu)$ and some $p>q$, allows us to study the asymptotic behaviour of $T_\Om(t)f$ as $t \to +\infty$ for $f\in L^p(\Om, \nu)$, $p>1$, and to relate it to the behaviour of the derivative $|D_HT_\Om(t)f|$ as $t\to +\infty$. This last result was already known in the finite dimensional setting for evolution operators associated to non-autonomous elliptic operator (see \cite{ALL13}).
These estimates are drawn in a more or less standard way: we have presented sketches of proofs (or even complete proofs) for the convenience of the reader.

Further consequences can be deduced, but these will be hopefully matter of other works.\\

\paragraph{\bf Acknowledgements} The authors are members of GNAMPA of the Italian Istituto Nazionale di Alta Matematica (INdAM). L.A. and S.F. have been partially supported by the INdAM-GNAMPA Project ``Equazioni e sistemi di equazioni di Kolmogorov in dimensione finita e non'' (2017). S.F. has been partially supported by the INdAM-GNAMPA Project ``Equazioni e sistemi di equazioni ellittiche e paraboliche a coefficienti illimitati'' (2018). D.P. has been partially supported by the INFN and by the INdAM-GNAMPA Projects ``Regolarit\`a massimale per alcuni operatori lineari ellittici degeneri'' (2017) and ``Teoria geometrica della misura in spazi metrici con misura'' (2018). The authors have been also partially supported by the PRIN project ``Deterministic and stochastic evolution equations'' of the Italian Ministry of Education MIUR.

\section*{Notations}
For any $k\geq 0$ and $n\in \N$, we denote by $C^k(\R^n)$ the space of continuous functions with continuous derivative up to the $[k]$-th order (here $[k]$ denotes the integer part of $k$) such that the $[k]$-th derivative is $(k-[k])$-H\"older continuous, if $k\notin \N$. We use the subscript ``$b$'' to denote the space of all functions in $C^k(\R^n)$ which are bounded together with all their derivatives up to the $[k]$-th order. $C^k_b(\Rn)$ is endowed  with the norm
\[
\|f\|_{C_b^k(\Rn)}:=\sum_{|\alpha|\le [k]}\|D^\alpha f\|_{\infty}+\sum_{|\alpha|=[k]}[D^\alpha f]_{k-[k]},
\]
where  $\norm{\cdot}_\infty$ denotes the sup-norm and, for any $\alpha\in (0,1)$, $[\cdot]_{\alpha}$ is the $\alpha$-H\"older seminorm. We use the subscript ``loc'' to denote the space of all $f\in C^{[k]}(\Rn)$ such that the derivatives of order $[k]$ are $(k-[k]$)-H\"older continuous in any compact subset of $\Rn$.
For any interval $J$ and $\alpha,\beta \ge 0$, we denote by $C^{\alpha,\beta}(J\times \Rn)$ the usual parabolic H\"older space.
The subscripts ``b'' and ``loc'' have the same meanings as above.

We also consider functions defined in infinite dimensional spaces. $X$ denotes a separable Hilbert space endowed with its norm $|\cdot|$ and inner product $\langle\cdot,\cdot\rangle$, while $\mathcal{L}(X)$ denotes the space of bounded linear operators from $X$ to itself, endowed with its operator norm $\norm{\cdot}_{\mathcal{L}(X)}$.

We define $C_b(X)$ to be the space of all functions $f:X\to \R$ which are continuous and bounded in $X$.
For any $k\in \N$, we denote by $C_b^k(X)$ the space of functions $f:X\to \R$ which have bounded and continuous Fr\'echet derivatives up to the order $k$ with norm
\[
\|f\|_{C_b^k(X)}:=\sum_{j=0}^{k}\|D^j f\|_{\infty},
\]
where $D^j$ denotes the $j$-th Fr\'echet derivative operator. Moreover if $f:X\ra \R$ is Lipschitz continuous we set $[f]_{\rm Lip}= \sup_{x,y \in X,\,x\neq y}\pa{|f(x)-f(y)||x-y|^{-1}}$.

For any $f:[0,+\infty)\times X\to \R$, once an orthonormal Hilbert basis $(v_i)_{i\in\N}$ has been fixed,
we use the symbols $D_tf, D_i f$ to denote, respectively, the time derivative of $f$ and the directional derivative of $f$ in the direction of $v_i$. We use the same notation in $\Rn$ where $D_i f$ denotes the directional derivative of $f$ along the $i$-th vector of the canonical basis of $\Rn$. Analogous meaning is given to the symbols $D_{ij}f$ and $D_{ijk}f$.

For any finite Radon measure $\mu$ on $X$ and $1\le p <\infty$, the set $L^p(X, \mu)$ consists of all measurable functions $f:X\to \R$ such that $\|f\|_{L^p(X,\mu)}^p:=\int_X |f|^p d\mu<+\infty$, while $L^\infty(X, \mu)$ is the space of all $\mu$-essentially bounded functions with norm $\|f\|_\infty= {\rm ess\,sup}_{x\in X}|f(x)|$. In a similar way we define the spaces $L^p(X, \mu;X)$ and $L^p(X, \mu; \mathcal{H}_2)$ where $\mathcal{H}_2$ is the space of Hilbert-Schmidt operators and the measurability is meant in Bochner's sense. With $p'$ we denote the conjugate exponent of $p$, i.e., $1/p+1/p'=1$, with the standard convention that $1'=\infty$.

 \makeatletter \@addtoreset{equation}{section}
\def\theequation{\thesection.\arabic{equation}}
\makeatother \makeatletter

\section{Assumptions and preliminary results}\label{prelim-sect}

We start this section by listing the hypotheses we assume throughout the paper.
\begin{hyp}\label{hyp_base}
Let assume that
\begin{enumerate}[\rm(i)]
\item $Q\in \mathcal{L}(X)$ is a self-adjoint and non-negative operator with ${\rm Ker}\, Q=\{0\}$;
\item $A:D(A)\subseteq X\ra X$ is a self-adjoint operator satisfying $\gen{Ax,x}\leq -\omega\abs{x}^2$ for every $x\in D(A)$ and some positive $\omega$;
\item $Qe^{tA}=e^{tA}Q$ for any $t\geq 0$;
\item ${\rm Tr}(-QA^{-1})<+\infty$.
\end{enumerate}
\end{hyp}
Under Hypotheses \ref{hyp_base} we can consider the Gaussian measure $\gamma$ with mean zero, covariance operator $Q_\infty:=-QA^{-1}$ and an orthonormal basis $(v_k)_{k\in\N}$ of $X$ such that
\begin{equation}\label{qinfty}
Q_\infty v_k=\lambda_k v_k,\qquad k\in\N,
\end{equation}
where $(\lambda_k)_{k\in\N}$ is the decreasing sequence of eigenvalues of $Q_\infty$.

The Cameron-Martin space $(H, |\cdot|_H)$, where
\[
H=\set{x\in X\tc \sum_{k=1}^{+\infty}\lambda_k^{-1}\gen{x,v_k}^2<+\infty},
\]
and $|\cdot|_H$ is the norm induced by the inner product $\gen{h,k}_H:=\langle Q_\infty ^{-1/2}h,Q_\infty ^{-1/2}k\rangle$, $h,k \in H$, is a Hilbert space which is densely embedded in $X$. Note that, as $H= Q_\infty^{1/2}X$, the sequence $(e_k)_{k\in \N}$, where $e_k=\sqrt{\lambda_k}v_k$ for any $k\in \N$, is an orthonormal basis of $H$.

We need to recall the definition of Lipschitz continuous function along the Cameron-Martin space $H$. If $Y$ is a Banach space with norm $\norm{\cdot}_Y$, a function $F:X\ra Y$ is said to be $H$-Lipschitz continuous if there exists a positive constant $C$ such that
\begin{gather}\label{costante di H-lip}
\norm{F(x+h)-F(x)}_Y\leq C\abs{h}_H,
\end{gather}
for every $h\in H$ and $\gamma$-a.e. $x\in X$ (see \cite[Section 4.5 and Section 5.11]{Bog} for the basic properties of $H$-Lipschitz continuous functions). We denote with $[F]_{H\text{-Lip}}$ the best constant $C$ appearing in \eqref{costante di H-lip}.

Now, we introduce a notion of derivative weaker than the classical Fr\'echet one. We say that
$f:X\rightarrow\R$ is $H$-differentiable at $x_0\in X$ if there exists $\ell\in H$ such that
\[
f(x_0+h)=f(x_0)+\gen{\ell,h}_H+o(|h|_H),\qquad\text{as $|h|_H\rightarrow0$.}
\]
In such a case we set $D_H f(x_0):=\ell$ and $D_i f(x_0):= \langle D_H f(x_0), e_i\rangle_H$ for any $i\in \N$. The derivative $D_H f(x_0)$ is called the \emph{Malliavin derivative} of $f$ at $(x_0)$. In a similar way we say that $f$ is twice $H$-differentiable at $x_0$  if $f$  is $H$-differentiable near $x_0$ and there exists $\mathcal B\in \mathcal{H}_2$ such that
\[
f(x_0+h)=f(x_0)+\gen{D_H f(x_0),h}_H+\frac{1}{2}\langle \mathcal B h,h\rangle_H+o(|h|^2_H),\qquad\text{as $|h|_H\rightarrow0$.}
\]
In such a case we set $D^2_H f(x_0):=\mathcal B$ and $D_{ij} f(x_0):= \langle D^2_H f(x_0)e_j, e_i\rangle_H$ for any $i,j\in \N$. We recall that if $f$ is twice $H$-differentiable at $x_0$, then $D_{ij}f(x_0)=D_{ji}f(x_0)$ for every $i,j\in\N$.

\begin{rmk}
{\rm If a function $f:X\rightarrow\R$ is (resp. twice) Fr\'echet differentiable at $x_0$ then it is (resp. twice) $H$-differentiable at $x_0$ and it holds $D_Hf(x_0)=Q_{\infty}^{1/2}Df(x_0)$,
(resp. $D^2_Hf(x_0)=Q_{\infty}^{1/2}D^2 f(x_0)Q_{\infty}^{1/2}$).}
\end{rmk}

For any $k\in\N\cup\set{\infty}$, we denote by $\mathcal{F}C_b^k(X)$, the space of cylindrical $C^k_b$ functions, i.e., the set of functions $f:X\to \R$ such that $f(x)=\varphi(\langle x, h_1\rangle,\ldots, \langle x, h_N\rangle)$ for some  $\varphi \in C_b^k(\R^N)$, $h_1,\ldots, h_N\in H$ and $N\in\N$.
By $\mathcal{F}C_b^k(X,H)$ we denote $H$-valued cylidrical $C_b^k$ functions with finite rank.

The Sobolev spaces in the sense of Malliavin $D^{1,p}(X,\gamma)$ and $D^{2,p}(X,\gamma)$ with $p\in[1,\infty)$, are defined as the completions of the \emph{smooth cylindrical functions} $\mathcal{F}C_b^\infty(X)$ in the norms
\begin{gather*}
\norm{f}_{D^{1,p}(X,\gamma)}:=\pa{\norm{f}^p_{L^p(X,\gamma)}+\int_X\abs{D_H f}_H^pd\gamma}^{\frac{1}{p}};\\
\norm{f}_{D^{2,p}(X,\gamma)}:=\pa{\norm{f}^p_{D^{1,p}(X,\gamma)}+\int_X|D_H^2 f|^p_{\mathcal{H}_2}d\gamma}^{\frac{1}{p}}.
\end{gather*}
This is equivalent to consider the domain of the closure of the gradient operator, defined on smooth cylindrical functions, in $L^p(X,\gamma)$. \\
We define a weighted Gaussian measure considering a function $U:X\ra\R$ that satisfies the following
\begin{hyp0}\label{ipotesi peso}
$U$ is a convex function which belongs to $C^2(X)\cap D^{1,q}(X,\gamma)$ for all $q\in[1,\infty)$.
\end{hyp0}

The convexity of the function $U$ guarantees that $U$ is bounded from below by a linear function, therefore it decreases at most linearly and by Fernique's theorem (see \cite[Theorem 2.8.5]{Bog}) $e^{-U}$ belongs to $L^1(X,\gamma)$. Then we can consider the finite log-concave measure
\begin{equation*}
\nu:= e^{-U}\gamma.
\end{equation*}
Notice that $\gamma$ and $\nu$ are equivalent measures, hence saying that a statement holds $\gamma$-a.e. is the same as saying that it holds $\nu$-a.e.
Moreover the fact that $U$ belongs to $D^{1,q}(X,\gamma)$ for any $q\in[1,\infty)$ allows us to
conclude that the operator $D_H:\mathcal{F}C^1_b(X)\ra L^p(X,\nu;H)$ is closable in $L^p(X,\nu)$, $p\in(1,\infty)$ and we may define the space $D^{1,p}(X,\nu)$, $p>1$, as the domain of its closure (still denoted by $D_H$). In a similar way we can define $D^{2,p}(X,\nu)$, $p\in(1,\infty)$ (for more details see \cite{CF16} and \cite{Fer15}). The Gaussian integration by parts formula $\int_X D_i f d\gamma=\frac{1}{\sqrt{\lambda_i}}\int_X \langle x,v_i\rangle f d\gamma$, which holds true for any  $f\in \mathcal{F}C^1_b(X)$ and $i\in\N$, yields that
\begin{gather}
\label{int by parts}\int_X \psi D_i \varphi d\nu+\int_X\varphi D_i\psi d\nu=\int_X \varphi\psi D_i U d\nu+\frac{1}{\sqrt{\lambda_i}}\int_X \langle x, v_i\rangle\varphi\psi d\nu,\qquad\;\,i \in \N,
\end{gather}
for any $\varphi, \psi \in \mathcal{F}C^1_b(X)$, hence by density for any $\varphi, \psi \in D^{1,p}(X,\nu),\ p\in(1,\infty)$.

In what follows $\Omega$ denotes an open convex subset of $X$. In this case the spaces $D^{1,p}(\Omega,\nu)$ and $D^{2,p}(\Omega,\nu)$, $p\in(1,\infty)$, can be defined in a similar way as in the whole space, thanks to the following result (proved in \cite{AMM} in the Gaussian case).

\begin{prop}
Let assume that Hypotheses $\ref{hyp_base}$ and $\ref{ipotesi peso}$ are satisfied and let $p\in(1,\infty)$ and $\Om$ be an open subset of $X$. Then the operators $D_H:\mathcal{F}C_b^\infty(\Omega)\ra L^p(\Omega,\nu; H)$ and
\begin{gather}\label{opra}
(D_H,D_H^2):\mathcal{F}C_b^\infty(\Omega)\times \mathcal{F}C_b^\infty(\Omega)\ra L^p(\Omega,\nu; H)\times L^p(\Omega,\nu; \mathcal{H}_2)
\end{gather}
are closable in $L^p(\Omega,\nu)$ and $L^p(\Omega,\nu)\times L^p(\Omega,\nu)$, respectively. Here $\mathcal{F}C_b^\infty(\Omega)$ is the space of the restriction to $\Omega$ of the functions in $\mathcal{F}C_b^\infty(X)$.
\end{prop}

\begin{proof}
We just prove that the operator $D_H:\mathcal{F}C_b^\infty(\Omega)\ra L^p(\Omega,\nu; H)$ is closable in $L^p(\Omega,\nu)$, since the proof that the operator defined in \eqref{opra} is closable in $L^p(\Omega,\nu)\times L^p(\Omega,\nu)$ is quite similar. By the linearity of the operator $D_H$ it is enough to prove that if $(f_k)_{k\in\N}\subseteq \mathcal{F}C^\infty_b(\Omega)$ is such that
\begin{align*}
&\phantom{D_H}\lim_{k\ra+\infty}f_k=0\qquad\text{ in }L^p(\Omega,\nu);\\
&\lim_{k\ra+\infty}D_H f_k=\Phi\qquad\text{ in }L^p(\Omega,\nu;H),
\end{align*}
then $\Phi=0$ $\nu$-a.e in $\Omega$.

By Lusin's theorem and standard arguments following from \cite{Man90}, the space ${\rm Lip}_c(\Omega)$  of the bounded Lipschitz functions $u$ defined on $X$ with bounded support such that ${\rm dist}(\supp u,\Omega^c)>0$ is dense in $L^p(\Omega,\nu)$. So it is enough to prove that $\int_{\Omega}\gen{\Phi,e_i}_Hu d\nu=0$, for every $i\in\N$ and $u\in {\rm Lip}_c(\Omega)$.

To this aim, let us fix $u\in {\rm Lip}_c(\Omega)$ and observe that, by the H\"older inequality, Hypothesis \ref{ipotesi peso} and the fact that $e^{-U}\in L^{q}(X,\gamma)$ for every $q\in[1,\infty)$, we get
\begin{align}\label{zampirone1}
\lim_{k\ra+\infty}\int_\Omega f_k D_i ud\nu&\leq [u]_{{\rm Lip}}[\nu(\Omega)]^{1/p'}
\lim_{k\ra+\infty}\|f_k\|_{L^p(\Om,\nu)}=0
\end{align}
and
\begin{equation}
\lim_{k\ra+\infty} \int_\Omega f_k uD_i Ud\nu \leq\norm{u}_\infty\|D_i U\|_{L^{p'q}(X,\gamma)}\| e^{-U}\|_{L^{q'}(X,\gamma)}^{1/p'} \lim_{k\ra+\infty}\|f_k\|_{L^p(\Om,\nu)}=0;\label{zampirone2}
\end{equation}
for every $i\in\N$ and $q\in(1,\infty)$.
Moreover, Fernique's theorem and the quoted hypotheses imply that
\begin{equation}
\lim_{k\ra+\infty} \int_\Omega f_k u\frac{\gen{x,v_i}}{\sqrt{\lambda_i}}d\nu
\leq\frac{\norm{u}_\infty}{\sqrt{\lambda_i}}
\pa{\int_X |\gen{x,v_i}|^{p'q}d\gamma}^{\frac{1}{p'q}}\| e^{-U}\|_{L^{q'}(X,\gamma)}^{1/p'} \lim_{k\ra+\infty}\|f_k\|_{L^p(\Om,\nu)}=0.\label{zampirone4}
\end{equation}
Now, we claim that $\int_\Omega \gen{\Phi,e_i}_H ud\nu=\lim_{k\ra+\infty}\int_X \tilde{u}D_i f_kd\nu$, where $\tilde{u}$ is the null extension of $u$ out of $\Omega$. Indeed, again by using the hypotheses listed above we get
\begin{align*}\lim_{k\ra+\infty} \int_\Omega u(D_if_k-\gen{\Phi,e_i}_H)d\nu&\leq\norm{u}_\infty[\nu(\Omega)]^{1/p'}
\!\!\!\lim_{k\ra+\infty}\pa{\int_\Omega |D_if_k-\gen{\Phi,e_i}_H|^pd\nu}^{1/p}\!\!\!=0.
\end{align*}
To conclude, let us observe that $\tilde{u}$ is Lipschitz continuous on $X$, so by the integration by parts formula \eqref{int by parts} and \eqref{zampirone1}-\eqref{zampirone4} we deduce
\begin{align*}
\int_\Omega \gen{\Phi,e_i}_H ud\nu&=\lim_{k\ra+\infty}\int_X \tilde{u}D_i f_kd\nu
\\
&=\lim_{k\ra+\infty}\int_X f_k\pa{-D_i \tilde{u}+\tilde{u}D_iU+\tilde{u}\frac{\gen{x,v_i}}{\sqrt{\lambda_i}}}d\nu\\
&=\lim_{k\ra+\infty}\int_\Omega f_k\pa{-D_i u+uD_iU+u\frac{\gen{x,v_i}}{\sqrt{\lambda_i}}}d\nu=0.
\end{align*}
This proves the claim.
\end{proof}
\noindent The spaces $D^{1,p}(\Omega,\nu;H)$, $p\in(1,\infty)$, are defined in a similar way, replacing smooth cylindrical functions with $H$-valued smooth cylindrical functions. We recall that if $F\in D^{1,p}(\Omega,\nu;H)$, then $D_H F$ belongs to $\mathcal{H}_2$.

In the sequel we consider boundary Cauchy problems defined in $\Om$ and we will need some continuity properties of the \emph{distance function along $H$}, $d_{\Om}:X\to [0,+\infty]$, defined by
\begin{gather*}
d_\Om(x):=\left\{\begin{array}{lc}
\inf\{|h|_H\,|\, h\in H\cap (\Omega-x)\}, \quad \quad& H\cap (\Omega-x)\neq\emptyset;\\
+\infty, & H\cap (\Omega-x)=\emptyset,
\end{array}\right.
\end{gather*}
for any  $x\in X$. In the following proposition we recall some results about the function $d_\Omega$ (see \cite[Theorems 2.8.5 \& 5.11.2]{Bog} and \cite[Section 3]{CF18}).

\begin{prop}\label{prop_dist}
Let $\Omega\subseteq X$ be an open convex set. Then $d_\Omega^2$ is $H$-differentiable and its Malliavin derivative is $H$-Lipschitz with $H$-Lipschitz constant less than or equal to $2$, i.e.,
\[|D_Hd_\Om^2(x+h)-D_Hd_{\Om}^2(x)|_H\le 2 |h|_H,\]
for any $h\in H$ and for $\nu$-a.e $x\in X$. Moreover $D_H^2d_\Omega^2$ exists $\nu$-a.e. in $X$ and $d_\Omega^2$ belongs to $D^{2,p}(X,\nu)$ for every $p\in[1,\infty)$.
\end{prop}

In order to prove our results we need further regularity of the second order Malliavin derivative of the distance function along $H$. More precisely, we assume that
\begin{hyp0}\label{ipo_convex}
$\Om$ is an open convex subset of $X$ such that $\nu(\partial\Omega)=0$ and $D_H^2d_\Omega^2$ is $H$-continuous $\gamma$-a.e. in $X$; i.e., for $\gamma$-a.e. $x\in X$
\[
\lim_{\abs{h}_H\ra 0}D^2_H d^2_\Omega(x+h)=D^2_H d^2_\Omega(x).
\]
\end{hyp0}

\begin{rmk}
{\rm We point out that there is a rather large class of subsets of $X$ satisfying Hypothesis \ref{ipo_convex}. For instance, by \cite{FR82} and \cite{Hol73}, if $\partial \Omega$ is (locally) a $C^2$-embedding in $X$ of an open subset of a hyperplane in $X$ and $\nu(\partial\Omega)=0$, then Hypothesis \ref{ipo_convex} is satisfied. Easy examples are:
\begin{enumerate}[{\rm (i)}]
\item every open ball and open ellipsoid of $X$;

\item every open hyperplane of $X$;

\item every set of the form $\Omega=\set{x\in X\tc G(x)< 0}$, where $G:X\ra\R$ is convex, belongs to $C^2(X)$ and $D_H G$ is non-zero at every point of $\partial \Omega$ (check \cite[Theorem 1(a)]{Hol73}).
\end{enumerate}}
\end{rmk}

An important tool in our analysis are the \emph{Moreau-Yosida approximants of $U$ along $H$}. We recall the main properties of this approximation and we refer to \cite[Section 12.4]{BauCom} for the classical theory and to \cite{ACF17,CF16,CF18} for the case considered here.

\begin{prop}\label{moreau-prop}
Let $f:X\ra\R\cup\set{+\infty}$ be a proper convex and lower semicontinuous function and denote by ${\rm dom}(f)=\set{x\in X\tc f(x)<+\infty}$. For any $\varepsilon>0$ and $x\in X$, let us consider
\begin{gather}\label{M env}
f_\eps(x):=\inf\set{f(x+h)+\frac{1}{2\eps}|h|_H^2\tc h\in H}.
\end{gather}
Then the following properties hold true:
\begin{enumerate}[\rm (i)]
\item $f_\eps(x)\leq f(x)$ for any $\eps>0$ and $x\in X$. Moreover $f_\eps(x)$ converges monotonically to $f(x)$ for any $x\in X$, as $\varepsilon\to 0^+$;

\item $f_\eps$ is $H$-differentiable in $X$ and $D_Hf_\eps$ is $H$-Lipschitz continuous in $X$;

\item $f_\eps$ belongs to $D^{2,p}(X,\gamma)$, whenever $f\in L^p(X,\gamma)$ for some $1\in[1,\infty)$;

\item if $x\in {\rm dom}(f)$ and $f$ belongs to $D^{1,p}(X,\gamma)$ for some $p\in[1,\infty)$, then $D_H f_\eps(x)$ converges to $D_H f(x)$ as $\eps\to 0^+$;

\item if $f\in C^2(X)\cap D^{2,p}(X,\gamma)$ for some $p\in[1,\infty)$ and $f$ is twice $H$-differentiable at every point $x\in{\rm dom}(f)$, then $D_H^2 f_\eps(x)$ exists and  converges to $D_H^2 f(x)$ as $\eps\to 0^+$, for any $x\in {\rm dom}(f)$. Furthermore $D^2_H f_\eps$ is $H$-continuous in $X$, i.e. $\lim_{|h|_H\to 0}D^2_H f_\eps(x+h)=D^2_h f_\eps(x)$ for any $x \in X$.
\end{enumerate}
\end{prop}

\paragraph{\bf Further notation} We now introduce some notations which will be largely used in the paper. For any $i,n \in \N$ and $x\in X$, we define
$x_i:=\sqrt{\lambda_i}\langle x,v_i\rangle$ and by $\Pi_n$ the projection $\Pi_n:X\to\Rn$, $\Pi_n x:=(x_1,\ldots, x_n)$. The function $\Sigma_n$ denotes the embedding $\Sigma_n:\Rn\to H$, $\Sigma_n \xi:= \sum_{k=1}^n \xi_k e_k$, for any $\xi \in\Rn$.
\noindent
Moreover, if $P_n:X\to H$ is defined by $P_n x:=\sum_{i=1}^n x_i e_i$ for any $x\in X$ and $n\in \N$, then the conditional expectation of $f$, $\E_n f$ defined as follows
\[
\E_nf(x):=\int_Xf\pa{P_nx+(I-P_n)y}d\gamma(y),\qquad\;\, f\in L^p(X,\gamma),\,\,p\in[1,\infty),
\]
enjoys some good continuity properties (see \cite[Corollary 3.5.2 and Proposition 5.4.5]{Bog} for a proof of the following result).

\begin{prop}\label{convergenza En f}
Assume that Hypotheses $\ref{hyp_base}$ hold true and let $1\leq p<+\infty$, $k\in\N$ and $f\in D^{k,p}(X,\gamma)$. Then $\E_n f$ belongs to $D^{k,p}(X,\gamma)$ and converges to $f$ in $D^{k,p}(X,\gamma)$ and pointwise $\gamma$-a.e. in $X$, as $n$ tends to $+\infty$. Moreover $\|\E_n f\|_{D^{k,p}(X,\gamma)}\leq\norm{f}_{D^{k,p}(X,\gamma)}$ and \[D_i\E_n f=\eqsys{\E_nD_i f & 1\leq i\leq n;\\
0 & i>n.}\]
\end{prop}

We conclude this section by recalling the main properties of the semigroup generated by the operator $L_\Omega$ in $L^2(\Om,\nu)$ defined as
\begin{align}
D(L_{\Omega})=\bigg\{u\in D^{1,2}(\Omega,\nu)\,\bigg|\,& \text{there exists\,\,} v_u\in L^2(\Omega,\nu)\text{ such that }\notag\\
&\left.\int_\Omega\gen{D_H u,D_H \varphi}_Hd\nu=-\int_\Omega v_u\varphi d\nu\text{ for every }\varphi\in\mathcal{F}C^\infty_b(\Om)\right\},\label{defn_LOm}
\end{align}
with $L_{\Omega}u:=v_u$ if $u\in D(L_{\Omega})$.

\begin{prop}\label{prop_TOm}
Under Hypotheses $\ref{hyp_base},\ref{ipotesi peso}$ and $\ref{ipo_convex}$, the following properties hold true.
\begin{enumerate}[\rm (i)]
\item For any $\lambda>0$ and $f\in L^2(\Om, \nu)$, the equation $\lambda u- L_{\Om}u=f$ in $\Om$ has a weak solution $u\in D^{1,2}(\Om, \nu)$, i.e., for every $\varphi\in D^{1,2}(\Omega,\nu)$ it holds
\[
\lambda\int_\Omega u\varphi d\nu+\int_\Omega \gen{D_H u,D_H \varphi}_H d\nu=\int_\Omega f\varphi d\nu.
\]
Moreover $u\in D^{2,2}(\Om, \nu)$ and the equation $\lambda u- L_{\Om}u=f$, $\lambda>0$, holds $\nu$-a.e. in $\Om$. Denoting by $R(\lambda,L_\Omega)$ the resolvent operator of $L_\Om$, the following estimates hold:
\begin{equation}\label{stime3}
\norm{R(\lambda,L_\Om)f}_{L^2(\Om,\nu)}\leq \frac{\norm{f}_{L^2(\Om,\nu)}}{\lambda},\quad \norm{D_HR(\lambda,L_\Om)f}_{L^2(\Om,\nu;H)}\leq \frac{\norm{f}_{L^2(\Om,\nu)}}{\sqrt{\lambda}},
\end{equation}
and
\begin{equation}\label{stime4}
\norm{D^2_H R(\lambda,L_\Om)f}_{L^2(\Om,\nu;\mathcal H_2)}\leq \sqrt{2} \norm{f}_{L^2(\Om,\nu)},
\end{equation}
Consequently $L_\Omega$ generates a bounded self-adjoint analytic semigroup $T_\Omega(t)$ in $L^2(\Omega,\nu)$.

\item $T_\Omega(t)$ can be extended to a positivity preserving contraction semigroup $($still denoted by $T_\Omega(t))$ in $L^p(\Omega,\nu)$ for every $1\leq p\leq +\infty$ and $t\geq 0$. In addition it is strongly continuous in $L^p(\Om, \nu)$ for any $p \in [1,+\infty)$.

\item If $f\in C_b(\Omega)$ has positive infimum in $\Om$, then $T_\Omega(t)f$ has a positive $\nu$-essential infimum, for any $t>0$.

\item For any convex function $\varphi:\R\to\R$,
\begin{equation}\label{jensen}
\varphi(T_\Omega(t)f)\leq T_\Omega(t)(\varphi\circ f),\qquad\;\, \nu\text{\rm -a.e.\,\,in\,\,} \Om,\,\, t>0, \,\,f\in C_b(\Om).
\end{equation}

\item For any $p\in (1,+\infty)$
\begin{equation}\label{holder}
T_\Omega(t)(fg)\leq(T_\Omega(t) \abs{f}^p)^{1/p}(T_\Omega(t)\abs{g}^{p'})^{1/{p'}}\qquad\;\, \nu\text{\rm -a.e.\,\,in\,\,} \Om,\,\,t>0, \,\,f,g\in C_b(\Om).
\end{equation}

\item For any $p\in[1,\infty)$, $f\in L^p(\Omega,\nu)$ and $g\in L^\infty(\Omega,\nu)$ it holds
\begin{gather*}
\int_\Omega f T_\Omega(t)gd\nu=\int_\Omega gT_\Omega(t)fd\nu,\qquad t>0.
\end{gather*}
\end{enumerate}
\end{prop}

\begin{proof}
(i) Inequalities \eqref{stime3} and \eqref{stime4} are proved in \cite[Theorem 1.3]{CF18}, while the last statement follows from the standard theory of semigroups.\\
(ii) It is a consequence of \cite[Theorem 2.14 and Corollary 2.18]{Ouh05}. Indeed by these results it is enough to prove the following two conditions:
\begin{enumerate}
\item if $u\in D^{1,2}(\Omega,\nu)$, then $\abs{u}\in D^{1,2}(\Omega,\nu)$ and $\int_\Omega\abs{D_H\abs{u}}_H^2d\nu\leq\int_\Omega\abs{D_Hu}_H^2d\nu$.
\item if $0\leq u\in D^{1,2}(\Omega,\nu)$, then $u\wedge \one:=\min\set{u,1}$ belongs to $D^{1,2}(\Omega,\nu)$ and
\begin{equation}\label{cappa}
\int_\Omega \abs{D_H (u\wedge \one)}_H^2d\nu\leq \int_\Omega\abs{D_Hu}_H^2d\nu.
\end{equation}
\end{enumerate}
Statement (1) follows from the fact that if $u$ belongs to $D^{1,2}(\Omega,\nu)$, then there exists a sequence $(u_n)_{n\in\N}\subseteq \mathcal{F}C^1_b(\Omega)$ converging to $u$ in $D^{1,2}(\Omega,\nu)$. It can be proved that the sequence $\tilde{u}_n=\sqrt{u_n^2+n^{-1}}$ converges to $|u|$ in $D^{1,2}(\Om,\nu)$ as $n\to +\infty$, namely $\abs{u}$ belongs to $D^{1,2}(\Omega,\nu)$. In addition $D_H\abs{u}={\rm sign}(u)D_Hu$ and $\int_\Omega\abs{D_H\abs{u}}_H^2d\nu=\int_\Omega\abs{D_Hu}_H^2d\nu$ (see \cite[Lemma 2.7]{DPL14} for further details).\\
To prove (2), as above we can consider a sequence $(u_n)_{n\in\N}\subseteq \mathcal{F}C^1_b(\Omega)$ converging to $u$ in $D^{1,2}(\Omega,\nu)$, as $n$ goes to infinity. Then, the sequence
\[\tilde{u}_n=\frac{1}{2}\pa{u_n+1-\sqrt{(u_n-1)^2+\frac{1}{n}}},\]
converges to $u\wedge\one$ as $n\to +\infty$, that is $u\wedge\one \in D^{1,2}(\Omega,\nu)$. Further,
\[
D_H(u\wedge \one)=\frac{1}{2}\pa{1-{\rm sign}(u-1)}D_Hu
\]
and \eqref{cappa} holds true (see \cite[Proposition 1.1]{Cap15} for more details). The strong continuity follows from \cite{Ouh05}.\\
(iii)-(v) Due to \cite[Theorem 4.3.5]{MR92} there is a Markov process $(\mathcal{Y},\mathcal{M},(X_t)_{t\geq 0}, (P_x)_{x\in X})$ such that $T_\Omega(t)f(x)={\bf E}_x(f(X_t))$ for $\nu$-a.e $x\in X$, where ${\bf E}_x$ denotes the expected value with respect to the probability measure $P_x$. We summarise here some of the main properties of the Markov process $(\mathcal{Y},\mathcal{M},(X_t)_{t\geq 0}, (P_x)_{x\in X})$ for the convenience of the reader:
\begin{itemize}
\item $(\mathcal{Y},\mathcal{M})$ is a measurable space;

\item there exists a filtration $(\mathcal{M}_t)_{t\ge 0}$ on $(\mathcal{Y},\mathcal{M})$ such that $(X_t)_{t\ge 0}$ is a $(\mathcal{M}_t)_{t\ge 0}$-adapted stochastic process;

\item $P_x$, $x\in X$, are probability measures on $(\mathcal{Y},\mathcal{M})$;

\item it holds $P_x[X_{s+t}\in A|\mathcal{M}_s]=P_{X_s}[X_t\in A]$ for all Borel set $A\subseteq X$, any $s,t\geq 0$ and for $P_x\text{-a.e.}$ $x\in X$.
\end{itemize}
We remark that in \cite[Chapter 4, Section 4(b)]{MR92} the authors study exactly the case we are in. The claims are easy consequences of the Jensen and H\"older inequalities.

\noindent (vi) Since $\mathcal{F}C^\infty_b(\Omega)$ is dense in $L^p(\Omega,\nu)$ for every $p\in[1,\infty)$, there exists  a sequence $(f_n)_{n\in\N}\subseteq \mathcal{F}C^\infty_b(\Om)$ such that $\lim_{n\ra+\infty}\|f_n-f\|_{L^p(\Omega,\nu)}$. By the self-adjointness of $T_\Omega$ in $L^2(\Omega,\nu)$ we get
\begin{gather*}
\int_\Omega f_n T_\Omega(t)gd\nu=\int_\Omega gT_\Omega(t)f_nd\nu,\qquad t>0.
\end{gather*}
By (ii) $T_\Omega(t)g\in L^\infty(\Omega,\nu)$, so letting $n$ to infinity we get the claim.
\end{proof}

\noindent If $\Omega=X$ the operator in \eqref{defn_LOm}, denoted by $L$, acts on smooth cylindrical functions $\varphi$ as follows
\begin{equation}\label{operator}
L\varphi:={\rm Tr}(D^2_H \varphi)-\sum_{i=1}^{\infty}\lambda_i^{-1}\langle x,e_i\rangle D_i\varphi -\langle D_HU,D_H\varphi\rangle_H, \quad\quad \nu\text{-a.e in } X,
\end{equation}
and it is symmetrised by the measure $\nu$, indeed
\begin{equation}
\label{weak form}\int_X\psi L\varphi d\nu=-\int_X\gen{D_H\varphi,D_H\psi}_Hd\nu,\qquad\,\, \varphi,\psi\in \mathcal{F}C^1_b(X).
\end{equation}

\noindent{\bf\em From now on we assume that Hypotheses \ref{hyp_base}, \ref{ipotesi peso} and \ref{ipo_convex} hold true.}

\section{An approximation result}

The main goal of this section is Theorem \ref{main theorem} which states that for any $f \in L^2(\Om, \nu)$ the function $T_{\Om}(t)f$ can be approximated in a suitable way by smooth enough functions written in terms of semigroups depending on two parameters $n$ and $\varepsilon$. These parameters take into account that the approximation procedure first reduces the problem from an infinite dimensional setting to a finite dimensional one and then, by using a penalisation argument, it allows to solve the problem in the domain $\Om$ throughout the solution of a suitable problem in the whole space.

In view of these facts we first recall some results about parabolic and elliptic problem with unbounded coefficients in finite dimension.

\subsection{Parabolic and elliptic equations in $\R^n$}\label{sect-finite}

In this subsection, we consider a convex function $\phi\in C^{2+\alpha}(\Rn)$ for some $\alpha \in(0,1)$ with bounded second derivatives and a second order differential operator $\mathcal{L}_{\phi}$ acting on smooth functions $v$ as follows
\begin{equation*}
\mathcal{L}_{\phi} v(\xi)=\Delta v(\xi)+\gen{\mathcal{B}\xi,Dv(\xi)}-\gen{D\phi(\xi),Dv(\xi)},\qquad\;\, \xi \in\Rn,
\end{equation*}
where $\mathcal{B}$ is a constant symmetric matrix such that $\gen{\mathcal{B}\xi,\xi}\leq -\beta\abs{\xi}^2$ for any $\xi \in \Rn$ and some $\beta> 0$.

It is known (see \cite[Chapter 1]{Lor17} and the reference therein) that for any $\varphi\in C_b(\Rn)$
there exists a unique bounded classical solution $v$ of problem
\begin{equation}
\left\{\begin{array}{ll}
D_t v(t,\xi)=\mathcal{L}_{\phi} v(t,\xi) &  t>0,\ \xi\in\R^n;\\
v(0,\xi)=\varphi(\xi), &  \xi\in\R^n.
\end{array}\right.\label{prob_finito_evol}
\end{equation}
Namely $v$ belongs to $C_b([0,+\infty)\times\Rn)\cap C^{1+\alpha/2, 2+\alpha}_{\rm loc}((0,+\infty)\times\Rn)$ and solves the Cauchy problem \eqref{prob_finito_evol}.
The uniqueness of $v$ is a consequence of the convexity of $\phi$ and of the existence of a Lyapunov function, i.e., a positive function $g\in C^2(\Rn)$ such that $\lim_{|\xi|\to +\infty}g(\xi)=+\infty$ and
\begin{equation}\label{Lyap}
(\mathcal L_\phi g)(\xi)-\lambda g(\xi)\le 0, \qquad\; \, \xi\in \Rn,
\end{equation}
for some $\lambda>0$. Indeed, taking $g(\xi)=\abs{\xi}^2$, $\xi \in\Rn$, we have
\begin{align*}
(\mathcal{L}_{\phi} g)(\xi)=&\, 2n+2\gen{\mathcal{B}\xi,\xi}-2\gen{D\phi(\xi),\xi}
\\
\le &\, 2n-2\beta\abs{\xi}^2-2\gen{D\phi(\xi)-D\phi(0),\xi}-2\gen{D\phi(0),\xi}
\\
\leq&\,  2n-2\beta\abs{\xi}^2+2\abs{D\phi(0)}\abs{\xi},
\end{align*}
where we have used that $\gen{D\phi(\xi)-D\phi(0),\xi}\geq 0$ for every $\xi\in\R^n$ so, clearly, we can find $\lambda$ such that inequality \eqref{Lyap} is satisfied.

In this way we can consider the semigroup $T_{\phi}(t)$ associated with $\mathcal{L}_\phi$ in $C_b(\Rn)$ and write $v(t,\xi)=(T_{\phi}(t)\varphi)(\xi)$ for any $t>0$ and $\xi\in \Rn$. It turns out that $T_\phi(t)$ is a positivity preserving contractive semigroup in $C_b(\Rn)$.

To pass from finite to infinite dimension, we prove and exploit suitable uniform gradient estimates independent of the dimension. More precisely, we prove a dimension-free uniform estimate for the gradient of $T_\phi(t)\varphi$, $\varphi\in C^1_b(\Rn)$. Such kind of estimates have already been proved for semigroups associated with more general operators (see \cite[Chapter 5]{Lor17} and the reference therein). However, since in all these estimates is not emphasised how and if the constants appearing depend on the dimension, we provide a sketch of the proofs (essentially based on the Bernstein method and the classical maximum principle) that allows us to verify that the constants are dimension-free.

\begin{prop}
The estimate
\begin{equation}\label{gra-est}
|D_\xi T_\phi(t) \varphi(\xi)|\leq \frac{\norm{\varphi}_\infty}{\sqrt{\beta t}}
\end{equation}
holds true for any $t>0$, $\xi\in\Rn$ and $\varphi\in C_b(\Rn)$. Here $\beta$ is the positive constant which bounds from below the quadratic form associated with $-\mathcal B$.
\end{prop}
\begin{proof}
It suffices to prove the claim for functions $\varphi\in C^{2+\alpha}_c(\Rn)$, i.e., the space of the functions in $C^{2+\alpha}(\R^n)$ with compact support. Indeed, if $\varphi\in C_b(\Rn)$ we can consider a sequence $(\varphi_m)_{m\in\N}$ converging to $\varphi$ locally uniformly as $m$ goes to infinity and use the fact that, up to a subsequence, $T_{\phi}(t)\varphi_m$ converges to $T_\phi(t)\varphi$ in $C^{1,2}_{\rm loc}((0,+\infty)\times \Rn)$, as $m$ goes to infinity (see \cite{Lor17}).
Moreover, taking advantage of the interior Schauder estimates (see \cite{Fri64}), we reduce ourselves to proving the claim for the solution $v_R$ of the homogeneous Neumann Cauchy problem associated with the equation $D_t v=\mathcal L_\phi v$ in $(0,T]\times B_R$, where $B_R$ is the open ball centered at the origin with radius $R$ large enough such that the support of $\varphi$ is contained in $B_R$. Indeed, once \eqref{gra-est} is proved for $v_R$, recalling that $v_R$ converges to $T_\phi(t)\varphi$ in $C^{1,2}_{\rm loc}((0,+\infty)\times \Rn)$ as $R\to +\infty$, we conclude. Therefore, let
$\varphi\in C^{2+\alpha}_c(\Rn)$ and $v_R$ be as specified above. Then the function
\[
z_R(t,\xi):=\abs{v_R(t,\xi)}^2+\beta t\abs{D_\xi v_R(t,\xi)}^2\qquad t>0,\,\xi\in B_R
\]
satisfies $z_R(0,\cdot)=\varphi^2$ in $B_R$, $\langle D_\xi v_R, \nu\rangle \le 0$ ($\nu$ is the unit normal vector) on $(0,T]\times \partial B_R$ and solves the equation
\begin{gather*}
D_t z_R-\mathcal{L}_{\phi}z_R=(\beta-1)\abs{D_\xi v_R}^2
+\gen{\mathcal{B}D_\xi v_R,D_\xi v_R}-\langle D_\xi^2\phi D_\xi v_R,D_\xi v_R\rangle
-\beta t|D_\xi^2 v_R|^2\le 0,
\end{gather*}
in  $(0,T]\times B_R$ (in the last inequality we have used the convexity of $\phi$ and the assumption on the matrix $\mathcal B$).
The classical maximum principle applied to the function $z_R-\|\varphi\|_\infty$ yields the claim in $(0,T]\times B_R$. The arbitrariness of $T$ allows us to extend the claim in the whole $(0,+\infty)\times B_R$.
\end{proof}

The contractivity of $T_{\phi}(t)$ in $C_b(\Rn)$ and estimate \eqref{gra-est} yield some dimension-free uniform estimates for the solution (and its gradient) of the elliptic equation
\begin{gather}\label{prob_finito_ellittico}
\lambda v-\mathcal{L}_{\phi} v= \varphi\in C^2_b(\Rn),\qquad\;\, \lambda >0.
\end{gather}
\begin{prop}\label{stime di schauder}
For any $\lambda >0$ there exists a unique bounded classical solution $v$ of the problem \eqref{prob_finito_ellittico}. Moreover $v$ satisfies
\begin{align}\label{prima_formula}
(i)\; \|v\|_\infty\leq\frac{1}{\lambda}\norm{\varphi}_\infty,\qquad\,\;
(ii)\;\|Dv\|_\infty\leq \sqrt{\frac{\pi}{\beta\lambda}}\norm{\varphi}_\infty.
\end{align}
In addition if $\phi\in C^3(\R^n)$, then $v$ belongs to $C^3_b(\R^n)$.
\end{prop}
\begin{proof}
Existence and estimates \eqref{prima_formula} are immediate consequences of the fact that
\begin{equation*}
v(\xi)=\int_0^{+\infty} e^{-\lambda t} (T_\phi(t)\varphi)(\xi)dt\qquad \xi\in\R^n.
\end{equation*}
(see \cite[Propositions 3.2 \& 3.4]{BF04} and \cite[Proposition 3.6]{Pri99}) and estimate \eqref{gra-est}.

Concerning the last statement, we just need to prove that the third order derivatives of $v$ are bounded. Indeed, the classical theory of elliptic equations guarantees that $v$ belongs to $C^{3}(\R^n)$. Moreover \cite[Theorem 1]{Lun98} yields that $u$ belongs to $C^{2+\theta}_b(\R^n)$ for every $0<\theta<1$ and
$\|v\|_{C^{2+\theta}_b(\Rn)}\le C\|\varphi\|_{C_b^\theta(\Rn)}$ for some positive constant $C$ independent of $\varphi$. Thus, we can differentiate \eqref{prob_finito_ellittico} and obtain
\begin{equation}\label{vetril}
\lambda D_jv-\mathcal{L}_{\phi} D_jv=D_j \varphi+(D^2\phi Dv)_j+(\mathcal{B} Dv)_j,\qquad\,\,
\end{equation}
for any $j=1, \ldots, n.$ Thus, taking into account that the right hand side of \eqref{vetril} is $\alpha$-H\"older continuous and bounded we can apply again \cite[Theorem 1]{Lun98} to deduce that $D_jv\in C^{2+\alpha}_b(\R^n)$ for every $j=1,\ldots,n$. In particular $v$ belongs to $C^3_b(\R^n)$.
\end{proof}

\subsection{Back to the infinite dimension}
Here we apply the results of the previous subsection with $\mathcal B={\rm diag}(-\lambda_1^{-1},\ldots,-\lambda_n^{-1})$ and $\beta=\lambda_1^{-1}$  (see \eqref{qinfty} for the definition of $(\lambda_i)_{i\in \N}$). Moreover, we focus on the  term $\langle D_HU, D_H\rangle_H$ in the operator $L$ in \eqref{operator}. We introduce some functions that, in some sense, reduce $U$ from infinite dimension to finite dimension and that contain a penalisation term which allows us to localise the problem in $\Om$.
More precisely, we define $\Phi_\eps:X\to \R$ and $\phi_{\ve,n}:\Rn\to\R$, respectively,  by
\[
\Phi_\eps(x):=U_\eps(x)+\frac{1}{2\eps}d^2_\Omega(x),\qquad
\phi_{\ve,n}(\xi):=(\E_n\Phi_\eps)(\Sigma_n\xi) \quad\; x\in X,\ \xi \in \Rn,\ n\in \N,\ \eps>0,
\]
where $U_\eps$ is the Moreau-Yosida approximation of $U$ along $H$ (see \eqref{M env}) and $\Sigma_n:\Rn\to X$ is the embedding function defined in Section \ref{prelim-sect}.

In order to apply the finite dimensional results obtained in Subsection \ref{sect-finite} we need also to regularise the function $\phi_{\ve,n}$. To do this we consider $\phi_{\ve,n,\eta}:\Rn\to \R$, the convolution of $\phi_{\ve,n}$ with a standard mollifier $\rho_\eta$.

First we state some properties of the functions just introduced. In the following statement we just need to show that the function $\phi_{\ve,n, \eta}$ belongs to $C^{2+\alpha}_b(\R^n)$ for any $\alpha \in(0,1)$.
\begin{lemm}\label{smooth}
For every $\eps,\eta>0$ and $n\in\N$, the function $\phi_{\ve,n, \eta}$ belongs to $C_b^\infty(\R^n)$.
\end{lemm}
\begin{proof}
Clearly, $\phi_{\ve,n,\eta}$ belongs to $C^\infty(\Rn)$. Let us show that $D^2\phi_{\eps,n,\eta}$ is bounded in
$\Rn$. Propositions \ref{prop_dist} and \ref{moreau-prop}(ii) guarantee that $\Phi_\eps$ is $H$-differentiable and $D_H \Phi_\eps$ is $H$-Lipschitz continuous in $X$. The same holds true in $\Rn$ for the functions $\phi_{\ve,n}$. Rademacher's theorem yields that $D\phi_{\ve,n}$ is differentiable $\mathcal{L}^n$-a.e. and $D^2\phi_{\ve,n}$ is $\mathcal{L}^n$-essentially bounded. This implies that $D^2 \phi_{\ve,n,\eta}$ are bounded in $\Rn$. With similar arguments it follows that $\phi_{\ve,n,\eta}\in C_b^{\infty}(\Rn)$.
\end{proof}

\begin{lemm}\label{costr_eta_n}
Let $\eps>0$. There exists a infinitesimal sequence $(\eta_n)_{n\in\N}$ such that
\begin{gather}
\lim_{n\ra+\infty}D_H\Phi_{\ve,n,\eta_n}=D_H\Phi_{\ve},\label{phone1}\\
\lim_{n\ra+\infty}D^2_H\Phi_{\ve,n,\eta_n}=D^2_H\Phi_{\ve},\label{phone2}
\end{gather}
where $\Phi_{\ve,n,\eta_n}(x):=\phi_{\ve,n,\eta_n}(\Pi_n x)$ for any $x\in X$. The limits in \eqref{phone1} and \eqref{phone2} are taken in $L^2(X,\nu_\eps;H)$ and $L^2(X,\nu_\eps;\mathcal{H}_2)$, respectively, and $\nu_\ve$ is the measure $e^{-\Phi_\eps}\gamma$.
\end{lemm}

\begin{proof}
Throughout this proof, for any $n\in\N$ and $x,y\in X$ we set $\Gamma_n(x,y):=P_nx+(I-P_n)y$. We start by proving that \eqref{phone1} holds true for every infinitesimal sequence $(\eta_n)_{n\in\N}$. To this aim, let $(\eta_n)_{n\in\N}$ be an infinitesimal sequence. Then
\begin{align}
\int_X\abs{D_H\Phi_{\ve}-D_H\Phi_{\ve,n,\eta_n}}_H^2d\nu_\eps\leq&\,
2\int_X\abs{D_H \Phi_{\ve}-D_H\E_n\Phi_{\ve}}_H^2
+\abs{D_H \E_n\Phi_{\ve}-D_H\Phi_{\ve,n,\eta_n}}_H^2 d\nu_\eps
\notag\\
\leq&\,  2\pa{\int_X e^{-p'\Phi_\eps}d\gamma}^{\frac{1}{p'}}
\pa{\int_X\abs{D_H \Phi_{\ve}-D_H\E_n\Phi_{\ve}}_H^{2p}d\gamma}^{\frac{1}{p}}
\notag\\
&+2\int_X\abs{D_H \E_n\Phi_{\ve}-D_H\Phi_{\ve,n,\eta_n}}_H^{2}d\nu_\eps.\label{ach}
\end{align}
Since $D_H d^2_\Om$ is $H$-Lipschitz continuous in $X$, the function $\Phi_{\ve}$ belongs to $D^{1,q}(X,\gamma)$ for$q\in[1,\infty)$. Thus Proposition \ref{convergenza En f} yields that the second line in \eqref{ach} vanishes as $n$ goes to infinity. Now
\begin{align}\label{ecotekne}
&\int_X\abs{D_H \E_n\Phi_{\ve}-D_H\Phi_{\ve,n,\eta_n}}_H^{2}d\nu_\eps\notag\\
&=\int_X\sum_{i=1}^n\Big(\int_XD_i\Phi_\eps(\Gamma_n(x,y))d\gamma(y)-\int_X\pa{\int_{\R^n}D_i\Phi_\eps\pa{\Gamma_n(x,y)-\eta_n(\Sigma_n\xi)}\rho(\xi)d\xi}d\gamma(y)\Big)^2d\nu_\eps(x)\notag\\
&\leq \int_X\int_{\R^n}\pa{\int_X\abs{D_H\Phi_\eps(\Gamma_n(x,y))-D_H\Phi_\eps\pa{\Gamma_n(x,y)-\eta_n(\Sigma_n(\xi)}}_H^2d\gamma(y)}\rho(\xi)d\xi d\nu_\eps(x)\notag\\
&\leq [D_H\Phi_\eps]_{H\text{-Lip}}^2\nu_\eps(X) \eta_n\int_{\R^n}\abs{\xi}^2\rho(\xi)d\xi \leq  [D_H\Phi_\eps]_{H\text{-Lip}}^2\nu_\eps(X)\eta_n,
\end{align}
and the right-hand side of \eqref{ecotekne} vanishes as $n\to +\infty$.

Now we prove \eqref{phone2}. Propositions \ref{prop_dist} and \ref{moreau-prop}(iii) guarantee that $\Phi_{\ve}$ belong to $D^{2,p}(X,\gamma)$ for any $p\in[1,\infty)$ and by Proposition \ref{convergenza En f} we immediately get that $D^2_H\E_n\Phi_{\ve}$ converges to $D^2_H \Phi_{\ve}$ in $L^2(X,\nu_{\ve};\mathcal H_2)$ as $n\to +\infty$.

In view of this fact, arguing as in \eqref{ach}, it remains to prove the existence of a vanishing sequence $(\eta_n)_{n\in \N}$ such that
$(D^2_H \E_n\Phi_{\ve}-D^2_H\Phi_{\ve,n,\eta_n})_{n\in\N}$ is infinitesimal in $L^2(X,\nu_{\ve};\mathcal H_2)$ as $n$ goes to infinity.

We start by showing that for any $n\in \N$ the function $D^2_H\Phi_{\ve,n,\eta}$ converges to $D^2_H\E_n\Phi_{\ve}$ in $L^2(X,\nu_{\ve};\mathcal H_2)$ as $\eta\to 0^+$.
To this aim, we can argue as in \eqref{ecotekne} and deduce that
\begin{align*}
&\int_X|D^2_H \E_n\Phi_{\ve}-D^2_H\Phi_{\ve,n,\eta}|_{\mathcal H_2}^{2}d\nu_\eps\\
&\leq \int_X\int_{\R^n}\pa{\int_X|D^2_H\Phi_\eps(\Gamma_n(x,y))-D^2_H\Phi_\eps\pa{\Gamma_n(x,y)-\eta(\Sigma_n\xi)}|_{\mathcal H_2}^2d\gamma(y)}\rho(\xi)d\xi d\nu_\eps(x).
\end{align*}
By Hypotheses \ref{ipotesi peso}, \ref{ipo_convex} and Proposition \ref{moreau-prop}(v), the function $D^2_H\Phi_\ve$ is $H$-continuous. This guarantees that the integrand function vanishes as $\eta\to 0$. Moreover, as $D^2_H\Phi_\eps$ is $\gamma$-essentially bounded in $X$, we can estimate the integrand function by a constant independent of $\eta$ and apply the dominated convergence theorem to conclude.

\noindent
Now, a diagonal argument yields an infinitesimal sequence satisfying \eqref{phone2} (and \eqref{phone1}, too).
\noindent
We start by letting $\eta_1$ be such that
\[\norm{D^2_H \E_1\Phi_{\ve}-D^2_H\Phi_{\ve,1,\eta_1}}_{L^2(X,\nu_\eps,\mathcal{H}_2)}<1.\]
Proceeding by induction, for every $n \ge 1$, we take $\eta_{n+1}$ in such a way that $\eta_{n+1}<\eta_n$ and
\[\norm{D^2_H \E_{n+1}\Phi_{\ve}-D^2_H\Phi_{\ve,n+1,\eta_{n+1}}}_{L^2(X,\nu_\eps,\mathcal{H}_2)}<\frac{1}{2^{n}}.\]
Thus, let $\eps>0$ and $\bar{n}\in\N$ be such that $1<2^{\bar{n}-1}\eps$ and $\norm{D^2_H\Phi_\eps-D^2_H \E_{n}\Phi_{\ve}}_{L^2(X,\nu_\eps,\mathcal{H}_2)}<\frac{\eps}{2}$ for any $n\geq \bar{n}$.  Then for $n\geq \bar{n}$
\begin{align*}
\norm{D^2_H\Phi_\eps-D^2_H\Phi_{\ve,n,\eta_n}}_{L^2(X,\nu_\eps,\mathcal{H}_2)}\leq&\norm{D^2_H\Phi_\eps-D^2_H\E_n\Phi_{\ve}}_{L^2(X,\nu_\eps,\mathcal{H}_2)}\\
&+\norm{D^2_H\E_n\Phi_\eps-D^2_H\Phi_{\ve,n,\eta_n}}_{L^2(X,\nu_\eps,\mathcal{H}_2)}\leq \frac{\eps}{2}+\frac{\eps}{2}=\eps.
\end{align*}
So the proof is complete.
\end{proof}

Now, let $f\in \mathcal F C^\infty_b(X)$ and $\psi \in C_b^\infty(\R^{n_0})$ for some $n_0 \in \N$ be such that $f(x)=\psi(\Pi_{n_0}x)$ for any $x \in X$. Proposition \ref{stime di schauder} and Lemma \ref{smooth} allow us to consider $v_{\ve,n,\eta_n}$, ($n \geq n_0$), the unique solution of \eqref{prob_finito_ellittico} with $\phi$ replaced by $\phi_{\ve,n,\eta_n}$ and $\varphi$ replaced by $\psi$.

In order to come back to the infinite dimensional setting we define
\begin{align*}
\Phi_{\ve,n}(x):=\phi_{\ve,n,\eta_n}(\Pi_n x),\qquad\; V_{\ve,n}(x):=v_{\ve,n,\eta_n}(\Pi_n x),\qquad\;\, x \in X,\ \varepsilon >0,\ n\geq n_0
\end{align*}
where $(\eta_n)_{n\in\N}$ is the sequence of Lemma \ref{costr_eta_n}. Now we consider the operator $L_\ve$ defined as
\begin{align*}
D(L_{\eps})=\bigg\{u\in D^{1,2}(X,\nu_\eps)\,\bigg|\,& \text{there exists\,\,} v_u\in L^2(X,\nu_\eps)\text{ such that }
\notag\\
&\left.\int_X\gen{D_H u,D_H \varphi}_Hd\nu_\eps=-\int_X v_u\varphi d\nu_\eps
\text{ for every }\varphi\in\mathcal{F}C^\infty_b(\Om)\right\},
\end{align*}
with $L_{\eps}u:=v_u$ if $u\in D(L_{\eps})$. We remark that $L_\eps$ acts on smooth cylindrical functions $\varphi$ as follows
\begin{equation}\label{operator_veps}
L_\eps\varphi={\rm Tr}(D^2_H \varphi)-\sum_{i=1}^{+\infty}\lambda_i^{-1}\langle x,e_i\rangle D_i\varphi-\langle D_H\Phi_\eps,D_H\varphi\rangle_H.
\end{equation}

\begin{rmk}\label{remark_Teps}
{\rm Note that formulas \eqref{int by parts} and \eqref{weak form} hold true also with $\nu$, $L$ and $U$ replaced by $\nu_\eps$, $L_\eps$ and $\Phi_\eps$, respectively. The same arguments listed after Hypothesis \ref{ipotesi peso} allow us to define the spaces $D^{k,p}(X,\nu_\eps)$ for any $\eps>0$, $p\in(1,\infty)$ and $k=1,2$. Moreover, if $(T_{\eps}(t))_{t\ge 0}$ is the analytic semigroup generated by the operator $L_\eps$ in $L^2(X,\nu_{\eps})$, then all the properties listed in Proposition \ref{prop_TOm} for $T_\Omega(t)$, hold true for $T_\eps(t)$, too.}
\end{rmk}

\begin{prop}\label{mannggia}
The function $V_{\ve,n}$ belongs to $\mathcal{F}C^3_b(X)$ and solves
\begin{align}
\lambda V_{\ve,n}-L_\ve V_{\ve,n}=f+\gen{D_H \Phi_\eps-D_H \Phi_{\ve,n},D_H V_{\ve,n}}_H=: f_n,\qquad\;\, \lambda>0.\label{Equazione soddisfatta da perturbazione}
\end{align}
Moreover $f_n$ converges to $f$ in $L^2(X,\nu_\eps)$ and $D_Hf_n$ converges to $D_Hf$ in $L^1(X,\nu_\eps,H)$, as $n$ goes to infinity.
\end{prop}

\begin{proof}
The fact that $V_{\ve,n}$ belongs to $\mathcal FC^3_b(X)$ follows from Proposition \ref{stime di schauder} and Lemma \ref{smooth}. In order to obtain \eqref{Equazione soddisfatta da perturbazione} we recall that $v_{\eps,n,\eta_n}(\xi)=V_{\eps,n}(\Sigma_n \xi)$ for any $\xi \in \R^n$. So we have
\begin{align*}
\lambda V_{\eps,n}(\Sigma_n \xi)-{\rm Tr}(D^2_H V_{\eps,n}(\Sigma_n \xi))
&+\sum_{i=1}^{+\infty} \lambda_i^{-1}\xi_i D_iV_{\eps,n}(\Sigma_n \xi)
\\
&+\gen{D_H\Phi_{\eps,n}(\Sigma_n \xi),D_H V_{\eps,n}(\Sigma_n \xi)}_H=\psi(\xi).
\end{align*}
Now adding and subtracting $L_\eps V_{\eps,n}(\Sigma_n\xi)$ (see \eqref{operator_veps}) and letting $\xi=\Pi_n x$ we get \eqref{Equazione soddisfatta da perturbazione}.
Observe that by Proposition \ref{stime di schauder} we also get the following estimate
\begin{equation}\label{Stime di schauder infinito dimensionali}
\norm{D_HV_{\eps,n}}_\infty\leq \sqrt{\frac{\lambda_1 \pi}{\lambda}}\norm{f}_\infty=:K.
\end{equation}

\noindent Using \eqref{Equazione soddisfatta da perturbazione} and \eqref{Stime di schauder infinito dimensionali} we get
\begin{align*}
\int_X\abs{f_n-f}^2d\nu_\ve=
\int_X\abs{\gen{D_H \Phi_\eps-D_H \Phi_{\ve,n},D_H V_{\ve,n}}}_H^2d\nu_\ve
\leq K^2\int_X\abs{D_H \Phi_\eps-D_H \Phi_{\ve,n}}_H^2d\nu_\ve.
\end{align*}
and by \eqref{phone1} we obtain that $f_n$ converges to $f$ in $L^2(X,\nu_\eps)$.

In order to prove the last part of the claim we first estimate $D_H^2 V_{\eps,n}$. Differentiating \eqref{Equazione soddisfatta da perturbazione} along $e_j$, multiplying the result by $D_j V_{\eps,n}$ and then summing up from $1$ to $n$, yields
\begin{align*}
\lambda\abs{D_H V_{\eps,n}}^2_H&-\sum_{j=1}^nD_jV_{\eps,n}L_\eps(D_j V_{\eps,n})
+\sum_{i=1}^n\lambda_i^{-1}(D_i V_{\eps,n})^2+\gen{D^2_H\Phi_\eps D_HV_{\eps,n},D_HV_{\eps,n}}_H
\\
&=\, \gen{D_Hf,D_HV_{\eps,n}}_H
+\gen{(D^2_H\Phi_\eps-D^2_H\Phi_{\eps,n})D_HV_{\eps,n},D_HV_{\eps,n}}_H
\\
&\quad \ +\gen{D_H^2V_{\eps,n}D_HV_{\eps,n},D_H\Phi_\eps-D_H\Phi_{\eps,n}}_H.
\end{align*}
Since $\lambda_i>0$ for every $i\in\N$, by the convexity of $\Phi_\eps$ we get
\begin{align}\label{prev}
-\sum_{j=1}^nD_jV_{\eps,n}L_\eps (D_j V_{\eps,n})\leq&\,  \gen{D_Hf,D_HV_{\eps,n}}_H+\gen{D_H^2V_{\eps,n}D_HV_{\eps,n},D_H\Phi_\eps-D_H\Phi_{\eps,n}}_H\notag\\
&+\gen{(D^2_H\Phi_\eps-D^2_H\Phi_{\eps,n})D_HV_{\eps,n},D_HV_{\eps,n}}_H.
\end{align}

\noindent Thus, integrating \eqref{prev} with respect to $\nu_\eps$ and using that
$$\int_X\gen{D_H u,D_H \varphi}_Hd\nu_\ve=-\int_X \varphi L_\ve u d\nu_\ve,\qquad\;\, u\in D(L_\ve),\,\varphi\in\mathcal{F}C^1_b(X),$$
we deduce
\begin{align*}
\int_X\abs{D^2_HV_{\eps,n}}_{\mathcal H_2}^2d\nu_{\eps}&\leq K\nu_\eps(X)\norm{D_Hf}_\infty+\sigma K\int_X\abs{D^2_HV_{\eps,n}}^2_{\mathcal H_2}d\nu_{\eps}\\
&+\frac{1}{4\sigma}K\int_X\abs{D_H\Phi_\eps-D_H\Phi_{\eps,n}}^2_Hd\nu_\eps+ K^2\pa{\int_X\abs{D^2_H\Phi_\eps-D^2_H\Phi_{\eps,n}}^2_{\mathcal H_2}d\nu_\eps}^{1/2}.
\end{align*}
for every $\sigma>0$. Choosing $\sigma=(2K)^{-1}$, we have
\begin{align*}
\frac{1}{2}\int_X\abs{D^2_HV_{\eps,n}}^2_{\mathcal H_2}d\nu_{\eps}\leq&\,
K\nu_\eps(X)\norm{D_Hf}_\infty+\frac{1}{2}K^2\int_X\abs{D\Phi_\eps-D\Phi_{\eps,n}}^2_Hd\nu_\eps
\\
&+K^2\pa{\int_X\abs{D^2\Phi_\eps-D^2\Phi_{\eps,n}}^2_{\mathcal H_2}d\nu_\eps}^{1/2}.
\end{align*}
Thanks to \eqref{phone1} and \eqref{phone2} there is a constant $C=C(K,\ve)>0$ such that $\|{D^2_HV_{\eps,n}}\|_{L^2(X,\nu_\ve;\mathcal H_2)}\leq C$ for every $n\in\N$.
\noindent
To complete the proof, we show that $D_Hf_n$ converges to $D_H f$ in $L^1(X,\nu_\ve;H)$. We have
\begin{align*}
\int_X\abs{D_H f_n-D_Hf}d\nu_\eps =\, &\int_X\abs{D_H\gen{D_H \Phi_\eps-D_H \Phi_{\ve,n},D_H V_{\ve,n}}_H}_Hd\nu_\ve
\\
\leq&\,  \int_X\abs{(D^2_H \Phi_\eps-D^2_H \Phi_{\ve,n})D_H V_{\ve,n}}_H+\abs{D^2_H V_{\ve,n}(D_H \Phi_\eps-D_H \Phi_{\ve,n})}_Hd\nu_\ve
\\
\leq&\,  K\|D^2_H \Phi_\eps-D^2_H \Phi_{\ve,n}\|_{L^2(X,\nu_\ve;\mathcal{H}_2)}\\
&+\|{D^2_HV_{\eps,n}}\|_{L^2(X,\nu_\ve;\mathcal H_2)}\|D_H \Phi_\eps-D_H \Phi_{\ve,n}\|_{L^2(X,\nu_\ve;H)}.
\end{align*}
So, being $\|{D^2_HV_{\eps,n}}\|_{L^2(X,\nu_\ve;\mathcal H_2)}$ bounded, the claim follows from  \eqref{phone1} and \eqref{phone2}.
\end{proof}

Proposition \ref{mannggia} and the Lumer--Phillips theorem yield that the resolvent set of $L_\ve$ in $L^2(X, \nu_\ve)$ contains the half-line $(0,+\infty)$. In addition, from \cite[Theorem 5.10]{CF16}, we get the following approximation result.

\begin{prop}\label{monedero-1}
For any $\eps>0$ and $f\in L^2(X,\nu_\eps)$ there exists a sequence $(f_n)_{n\in\N}$ in $D^{1,2}(X,\nu_\eps)$ such that $R(\lambda,L_\eps)f_n$ belongs to $\mathcal{F}C^3_b(X)$ for every $n\in\N$ and
\begin{gather*}
\lim_{n\ra+\infty}\norm{R(\lambda,L_\eps)f_n-R(\lambda,L_\eps)f}_{D^{2,2}(X,\nu_\eps)}=0,\qquad\;\, \lambda>0,
\end{gather*}
where $R(\lambda,L_\eps)$ is the resolvent operator of $L_\eps$ $($see \eqref{operator_veps}$)$. In addition
\begin{equation}\label{stime1}
\norm{R(\lambda,L_\eps)f}_{L^2(X,\nu_\eps)}\leq \frac{1}{\lambda} \norm{f}_{L^2(X,\nu_\eps)},\qquad \norm{D_HR(\lambda,L_\eps)f}_{L^2(X,\nu_\eps;H)}\leq \frac{1}{\sqrt{\lambda}} \norm{f}_{L^2(X,\nu_\eps)},
\end{equation}
and
\begin{equation}\label{stime2}
\norm{D^2_H R(\lambda,L_\eps)f}_{L^2(X,\nu_\eps;\mathcal H_2)}\leq \sqrt{2} \norm{f}_{L^2(X,\nu_\eps)}.
\end{equation}
\end{prop}

Now, we are ready to prove the main theorem of this section.

\begin{thm}\label{main theorem}
The following statements hold true.
\begin{enumerate}[\rm(i)]
\item For any $\eps>0$ and $f\in L^2(X,\nu_\eps)$ it holds that
\begin{gather*}
\lim_{n\ra+\infty}\norm{T_{\eps}(t)f_n-T_\eps(t) f}_{D^{2,2}(X,\nu_\eps)}=0,\qquad\;\, t>0,
\end{gather*}
where $(f_n)_{n\in\N}$ is the sequence defined in Proposition \ref{monedero-1}.
Furthermore $T_\eps(t)f_n$ belongs to $\mathcal{F}C^3_b(X)$. In addition, if $f\in D^{1,2}(X,\nu_\ve)$  then $D_Hf_n$ converges to $D_Hf$ in $L^1(X, \nu_\ve; H)$, as $n$ goes to infinity.
\item for any $f\in L^2(\Omega,\nu)$ there exists an infinitesimal sequence $(\eps_n)_{n\in\N}$ such that
$T_{\eps_n}(t) \tilde{f}$ weakly converges to $T_\Omega(t) f$ in $D^{2,2}(\Omega,\nu)$, where $\tilde{f}$ is any $L^2$-extension of $f$ to $X$.
\end{enumerate}
\end{thm}
\begin{proof}
The analyticity of the semigroups $T_{\Om}(t)$ and $T_{\ve}(t)$ in $L^2(\Om, \nu)$ and $L^2(X, \nu_\ve)$, respectively and the decay estimates \eqref{stime3},\eqref{stime4}, \eqref{stime1} and \eqref{stime2} (and Remark \ref{remark_Teps}) allow us to write the following representation formulas
\begin{align}\label{for_res}
&D^j_HT_\eps(t)f=\frac{1}{2\pi i}\int_{\sigma} e^{\lambda t} D^j_HR(\lambda,L_\eps)fd\lambda,& t>0,\,f\in L^2(X,\nu_\eps),\\
&D^j_HT_\Omega(t)f=\frac{1}{2\pi i}\int_{\sigma'} e^{\lambda t} D^j_HR(\lambda,L_\Omega)fd\lambda,& t>0,\,f\in L^2(\Omega,\nu),\notag
\end{align}
for any $j=0,1,2$, where $\sigma$ (resp. $\sigma'$) is a smooth (unbounded) curve in $\CC$ which leaves on the left a sector containing the spectrum of $L_{\eps}$ (resp. $L_\Om$).\\
(i) For any $j=0,1,2$ we have
\begin{align}\label{fro}
\int_X |D^j_HT_{\eps}(t)f_n&-D^j_HT_\eps(t)f|_j^2d\nu_\eps=
\frac{1}{4\pi^2}\int_X\abs{\int_{\sigma}e^{\lambda t}\pa{D^j_HR(\lambda,L_{\eps})f_n d\lambda-D^j_HR(\lambda,L_{\eps})f }d\lambda}_j^2d\nu_\eps
\notag\\
& \leq \frac{K(\sigma,t)}{4\pi^2}\int_{\sigma}\int_X e^{\lambda t}|D^j_HR(\lambda,L_{\eps})f_n-D^j_HR(\lambda,L_{\eps})f|_j^2 d\nu_\eps d\lambda,
\end{align}
where $|\cdot|_j$ denotes the norm in $\R,\ H,\ \mathcal{H}_2$respectively and $K(\sigma,t)=\int_\sigma e^{\lambda t}d\lambda$. We conclude observing that, by the dominated convergence theorem and the results in Proposition \ref{monedero-1}, the right hand side of \eqref{fro} vanishes as $n$ goes to infinity. The furthermore part is consequence of Proposition \ref{monedero-1} and the integral representation formula \eqref{for_res}. Finally, the last assertion is an immediate consequence of Proposition \ref{mannggia}.

\noindent (ii) Since $U(x)\ge \Phi_\ve(x)$ for any $x \in \Om$, by using \eqref{stime1} and \eqref{stime2} we immediately deduce that for any vanishing sequence $(\ve_n)$ and for any $f \in L^2(\Om, \nu)$ the sequence $(R(\lambda, L_{\ve_n})\tilde{f})$ is bounded in $D^{2,2}(\Om, \nu)$. A compactness argument yields that there exists a subsequence of $(\ve_n)$ (still denoted by $(\ve_n)$) such that $R(\lambda, L_{\ve_n})\tilde{f}$ weakly converges to an element $u \in D^{2,2}(\Om, \nu)$, as $n$ goes to infinity. From \cite[Theorem 5.3]{CF18} it follows that $u= R(\lambda, L_\Om)f$. Now, the proof proceeds as in (i). Indeed, for any $f,g \in L^2(\Om, \nu)$ we have
\begin{align*}
\int_{\Om}(T_{\ve_n}(t)\tilde{f})g d\nu=\, & \frac{1}{2\pi i}\int_\Om\int_{\sigma} e^{\lambda t} (R(\lambda,L_{\ve_n})\tilde{f})gd\lambda d\nu= \frac{1}{2\pi i}\int_{\sigma} e^{\lambda t}\int_\Om (R(\lambda,L_{\ve_n})\tilde{f})g d\nu d\lambda,
\end{align*}
Now, arguing as in (i), by the dominated convergence theorem we deduce
\[
\lim_{n \to +\infty}\int_{\Om}(T_{\ve_n}(t)\tilde{f})g d\nu=
\frac{1}{2\pi i}\int_\Om\int_{\sigma} e^{\lambda t} (R(\lambda,L_\Om))f)gd\lambda d\nu
= \int_{\Om}(T_{\Om}(t)f)g d\nu.
\]
In a similar fashion it is possible to prove that $D_HT_{\eps_n}(t)\tilde{f}$ weakly converges to $D_HT_\Omega(t)f$ in $L^2(\Om, \nu;H)$ and that $D^2_HT_{\eps_n}(t)\tilde{f}$ weakly converges to $D^2_HT_\Omega(t)f$ in $L^2(\Om, \nu;\mathcal{H}_2)$.
\end{proof}

\section{Pointwise gradient estimates}
In this section we prove some pointwise gradient estimates for $T_\Om(t)$. As already observed in the Introduction, these estimates are interesting since firstly, they represent a generalisation to what it is known in literature and, secondly, they allow to deduce many properties of $T_{\Om}(t)$ and of the associated invariant measure $\nu$, as the results in Section \ref{SecLogSob} show.

\begin{thm}\label{pointwise}
For any $p\in [1,+\infty)$ and $f\in D^{1,p}(\Om,\nu)$
\begin{equation}\label{point}
|D_H T_\Om(t)f|_H^p\le e^{- p \lambda_1^{-1} t}(T_\Om(t)|D_H f|^p_H),\qquad\, t>0,\ \nu\text{\rm -a.e. in }\Om.
\end{equation}
\end{thm}
\begin{proof}
First we prove the claim with $p=1$ and $f\in \mathcal F C^\infty_b(\Om)$. Next we address to the general case.\\
Let $f \in \mathcal{F}C_b^\infty(X)\subseteq D^{1,2}(X,\nu)(\subseteq D^{1,2}(X,\nu_\ve), \text{ for any }\eps>0)$ and $g$ a bounded, continuous and positive function. To overcome the lack of regularity of the function $|D_H T_\Om(t)f|_H$ at its zeros, we replace it by $\eta_\sigma(|D_H T_\Om(t)f|_H^2)$ where $\eta_\sigma:[0,+\infty)\ra[0,+\infty)$ is  the concave and smooth function defined by $\eta_\sigma(\xi):=\sqrt{\sigma+\xi}-\sqrt{\sigma}$ for any $\xi\ge 0$ and $\sigma>0$. Note that $\eta_\sigma$ is Lipschitz continuous in $[0,+\infty)$ and satisfies
\begin{equation}\label{prop_eta}
\textrm{(i)}\,\eta_\sigma(\xi)\le \sqrt{\xi},\qquad\,\,\textrm{(ii)}\,\xi\eta'_\sigma(\xi)\ge \frac{1}{2}\eta_\sigma(\xi),\qquad\;\,\textrm{(iii)}\,\eta'_\sigma(\xi)+2\xi\eta''_\sigma(\xi)\ge 0,
\end{equation}
for any $\xi \ge 0$ and $\sigma >0$.

To proceed further, we need to control the third order spatial derivatives of $T_\Om(t)f$. Since we are not able to do that directly on $T_\Om(t)f$, we replace it by the double indexed approximating sequence $(T_{\varepsilon_k}(t)f_n)_{n,k\in\N}$ where the sequences $(\ve_k)_{k\in\N}$ and $(f_n)_{n\in\N}$ are as in Theorem \ref{main theorem}. More precisely $\ve_k$ vanishes as $k$ goes to infinity, $(T_{\eps_k}(t)f_n)\subseteq \mathcal{F}C^3_b(X)$ and $D_H f_n$ converges to $D_Hf$ in $L^1(X,\nu_\ve;H)$ as $n\to +\infty$.
Hence, for any $t>0$, $\tau, s \in [0,t]$ and $k,n \in \N$ we define
\[
w_{\tau}^{\ve_k,n}:=|D_H u_{\ve_k,n}(\tau)|_H^2,\qquad\;\, G(s)=G_{\sigma,h}^{\ve_k,n}(s):=\int_{X}\eta_\sigma(w_{t-s}^{\ve_k,n})T_{\ve_k}(s)gd\nu_{\ve_k},
\]
where, to simplify the notation, we have set $u_{\ve_k,n}:=T_{\varepsilon_k}(\cdot)f_n$ for any $k,n\in \N$.
Recall that $\nu_{\ve_k}= e^{-\Phi_{\ve_k}}\gamma$ is the invariant measure associated with $T_{\ve_k}(t)$ and that by the definition of the operator $L_{\ve_k}$ we get
\begin{equation}\label{inbyp}
\int_X \psi_1 L_{\ve_k}\psi_2 d \nu_{\ve_k}=-\int_X \langle D_H \psi_1,D_H \psi_2\rangle_H d\nu_{\ve_k},
\end{equation}
with $\psi_1 \in D^{1,2}(X,\nu_{\ve_k})$ and $\psi_2\in D(L_{\eps_k})= D^{2,2}(X, \nu_{\ve_k})$ (see \cite[Theorem 6.2]{CF16} for the characterisation of the domain of $D(L_{\eps_k})$).
Theorem \ref{main theorem} guarantees that, for every $t\geq 0$, the function $u_{\ve_k,n}(t,\cdot)$ belongs to $\mathcal{F}C^3_b(X)$ and, as consequence, that $G$ is differentiable in $(0,t)$. Thus, taking into account that
\begin{align*}\frac{d}{ds}\eta_\sigma(w_{t-s}^{\ve_k,n})&=\eta'_\sigma(w_{t-s}^{\ve_k,n})\frac{d}{ds}|D_H u_{\ve_k,n}(t-s)|_H^2\\
&= -2\eta'_\sigma(w_{t-s}^{\ve_k,n})\langle D_H u_{\ve_k,n}(t-s),D_H (L_{\ve_k}u_{\ve_k,n}(t-s))\rangle_H
\end{align*}
and using \eqref{inbyp} twice, we deduce
\begin{align}\label{derivative}
G'(s)=\, &  -2\int_X\eta'_{\sigma}(w_{t-s}^{\ve_k, n})\langle D_H u_{\ve_k,n}(t-s),D_H (L_{\ve_k}u_{\ve_k,n}(t))\rangle_H T_{\ve_k}(s)g d\nu_{\ve_k}\notag\\
&-\int_X\eta_{\sigma}'(w_{t-s}^{\ve_k, n})\langle D_H T_{\ve_k}(s)g, D_H w_{t-s}^{\ve_k, n}\rangle_H d\nu_{\ve_k}\notag\\
=\, & -2\int_X\eta'_{\sigma}(w_{t-s}^{\ve_k, n})\langle D_H u_{\ve_k,n}(t-s),D_H (L_{\ve_k}u_{\ve_k,n}(t-s))\rangle_H T_{\ve_k}(s)g d\nu_{\ve_k}\notag\\
&-\int_X\langle D_H (\eta_{\sigma}'(w_{t-s}^{\ve_k, n})T_{\ve_k}(s)g), D_H w_{t-s}^{\ve_k, n}\rangle_H d\nu_{\ve_k}+\int_X \eta_{\sigma}''(w_{t-s}^{\ve_k, n})T_{\ve_k}(s)g|D_H w_{t-s}^{\ve_k, n}|_H^2 d\nu_{\ve_k}\notag\\
=\, & -2\int_X\eta'_{\sigma}(w_{t-s}^{\ve_k, n})\langle D_H u_{\ve_k,n}(t-s),D_H (L_{\ve_k}u_{\ve_k,n}(t-s))\rangle_H T_{\ve_k}(s)g d\nu_{\ve_k}\notag\\
&+\int_X\eta_{\sigma}'(w_{t-s}^{\ve_k, n})T_{\ve_k}(s)g L_{\ve_k} (w_{t-s}^{\ve_k, n}) d\nu_{\ve_k}+\int_X \eta_{\sigma}''(w_{t-s}^{\ve_k, n})T_{\ve_k}(s)g|D_H w_{t-s}^{\ve_k, n}|_H^2 d\nu_{\ve_k}\notag\\
=\, &2\!\int_X\!\!\eta'_{\sigma}(w_{t-s}^{\ve_k, n})T_{\ve_k}(s)g\!\!\left(\frac{1}{2}L_{\ve_k} (w_{t-s}^{\ve_k, n})
-\langle D_H u_{\ve_k,n}(t-s),\!D_H (L_{\ve_k}u_{\ve_k,n}(t-s))\rangle_H\right)\!\!d\nu_{\ve_k}\notag\\
&+\int_X \eta_{\sigma}''(w_{t-s}^{\ve_k, n})T_{\ve_k}(s)g|D_H w_{t-s}^{\ve_k, n}|_H^2d\nu_{\ve_k}.
\end{align}
Now, a straightforward computation and Hypothesis \ref{hyp_base} yield that
\begin{align*}
&\frac{1}{2}L_{\ve_k} (w_{\cdot}^{\ve_k, n})
-\langle D_H u_{\ve_k,n},D_H (L_{\ve_k}u_{\ve_k,n})\rangle_H
\notag\\
&=|D^2_H u_{\ve_k,n}|_{\mathcal H_2}^2+\sum_{i=1}^{+\infty} \lambda_i^{-1}(D_i u_{\ve_k,n})^2 +\langle D^2_H \Phi_{\ve_k} D_H u_{\ve_k,n}, D_H u_{\ve_k,n}\rangle_H
\notag\\
& \ge |D^2_H u_{\ve_k,n}|_{\mathcal H_2}^2+ \lambda_1^{-1}|D_H u_{\ve_k,n}|_H^2+\langle D^2_H \Phi_{\ve_k} D_H u_{\ve_k,n}, D_H u_{\ve_k,n}\rangle_H.
\end{align*}
In addition it is easy to prove that
\begin{align}\label{b}
|D_H w_{\cdot}^{\ve_k, n}|_{\mathcal H_2}^2= 4 |D^2_H u_{\ve_k,n}D_H u_{\ve_k,n}|_{\mathcal H_2}^2\le 4|D^2_H u_{\ve_k,n}|_{\mathcal H_2}^2w_{\cdot}^{\ve_k, n}.
\end{align}
Thus, using \eqref{derivative} and \eqref{b}, taking into account the convexity of $\Om$ and $U$ and the fact that $\eta_\sigma''\le 0$ in $(0,+\infty)$ we deduce that
\begin{align*}
G'(s)\geq&\,  2\int_X [\eta'_\sigma(w_{t-s}^{\ve_k,n})+2 \eta_{\sigma}''(w_{t-s}^{\ve_k, n})w_{t-s}^{\ve_k, n}]T_{\ve_k}(s)g|D^2_H u_{\ve_k,n}(t-s)|_{\mathcal H_2}^2d\nu_{\ve_k}
\notag\\
&+2\lambda_1^{-1}\int_X\eta'_{\sigma}(w_{t-s}^{\ve_k, n})w_{t-s}^{\ve_k, n}T_{\ve_k}(s)g d\nu_{\ve_k}\notag\\
\geq&\,  \lambda_1^{-1}\int_X \eta_{\sigma}(w_{t-s}^{\ve_k, n})T_{\ve_k}(s)g d\nu_{\ve_k}= \lambda_1^{-1} G(s),
\end{align*}
where in the last inequality we have used also \eqref{prop_eta}(ii)-(iii). Integrating the previous estimate with respect to $s$ in $(0,t)$ we get $G(0)\le e^{-\lambda_1^{-1}t}G(t)$ and letting $\sigma\to 0$ we deduce
\begin{align}\label{toloc}\int_X  |D_H u_{\ve_k,n}(t)|_Hg d\nu_{\ve_k}\le &e^{- \lambda_1^{-1} t} \int_X |D_H f_n|_HT_{\ve_k}(t)g d\nu_{\ve_k}.
\end{align}
Proposition \ref{prop_TOm}(vi), Remark \ref{remark_Teps} and formula \eqref{toloc} imply
\begin{align}\label{biglietto}
\int_X |D_H u_{\ve_k,n}(t)|_Hg d\nu_{\ve_k}\le &e^{-\lambda_1^{-1}  t}\int_X (T_{\ve_k}(t)|D_H f_n|_H)g d\nu_{\ve_k}.
\end{align}
Since formula \eqref{biglietto} holds true for every positive, bounded and continuous function $g$ and the measures $\nu_\eps$ and $\nu$ are equivalent, we get $|D_H u_{\ve_k,n}(t)|_H\leq e^{-\lambda_1^{-1}  t}T_{\ve_k}(t)|D_H f_n|_H$, $\nu$-a.e. in $X$ for every $k,n\in\N$ and $t\geq 0$. From Theorem \ref{main theorem}, up to subsequences, we get that $|D_H u_{\ve_k,n}(t)|_H$ and $T_{\ve_k}(t)|D_H f_n|_H$ pointwise converge $\nu$-a.e. in $\Om$ to $|D_H T_\Om(t)f|_H$ and $T_\Om(t)|D_H f|_H$ respectively, as $k,n\to +\infty$. This yields \eqref{point} with $p=1$ and $f\in\mathcal{F}C^\infty_b(\Omega)$. Formula \eqref{jensen} allows to extend the previous estimate to any $p\in(1,\infty)$.\\
Finally, let $f\in D^{1,p}(\Omega,\nu)$ and let $(g_n)_{n\in\N}\subseteq\mathcal{F}C^\infty_b(\Omega)$ be a sequence converging to $f$ in $D^{1,p}(\Omega,\nu)$ and pointwise $\nu$-a.e. in $\Om$. Formula \eqref{point} with $f$  replaced by $g_n-g_m$ and the invariance of $\nu$ with respect to $T_\Om(t)$ give that the sequence $(D_H T_{\Om}(t)g_n)_{n\in\N}$ is a Cauchy sequence in $L^p(\Om,\nu;H)$. Since $ T_{\Om}(t)g_n$ converges to $ T_{\Om}(t)f$ in $L^p(\Om ,\nu)$ and the operator $D_H$ is closable in $L^p(\Om,\nu)$, we obtain that $ D_HT_{\Om}(t)g_n$ converges to $D_H T_{\Om}(t)f$ in $L^p(\Om,\nu;H)$. Writing \eqref{point} with $f$ replaced by $g_n$ and letting $n\to +\infty$ yield the claim in the general case.
\end{proof}

\begin{cor}
For any $p\in (1,+\infty)$ and $f\in D^{1,p}(X,\nu)$, it holds that
\begin{gather*}
\lim_{t\ra 0^+}\|D_H T_\Om(t)f\|_{L^p(\Om,\nu;H)}=\|D_H f\|_{L^p(\Om,\nu;H)}.
\end{gather*}
\end{cor}
\begin{proof}
By the strong continuity of $T_\Omega(t)$ and the lower semicontinuity of the $L^p$-norm of the gradient we have
\[
\|D_H f\|_{L^p(\Om,\nu;H)}\le \liminf_{t\to 0^+}\|D_H T_{\Om}(t)f\|_{L^p(\Om,\nu;H)}.
\]
Hence, by \eqref{point}
\begin{align*}
\int_{\Om}|D_H f|_H^{p}d\nu&\le \liminf_{t \to 0^+}\int_\Om |D_H T_\Om(t)f|_H^{p} d\nu\le \limsup_{t \to 0^+}\int_\Om |D_H T_\Om(t)f|_H^{p} d\nu
\notag\\
&\le \lim_{t \to 0^+} e^{-p\lambda_1^{-1} t}\int_\Om T_\Om(t)|D_H f|_H^{p} d\nu
=\int_\Om |D_H f|_H^{p}d\nu
\end{align*}
and the proof is complete.
\end{proof}

Now we prove a pointwise gradient-function estimate for $T_{\Om}(t)f$ whenever $f\in L^p(\Om,\nu)$ and $p\in(1,\infty)$. The proof is similar to \cite[Theorem 6.2.2]{Lor17}, however it cannot be directly adapted to $T_{\Om}(t)f$ in view of the possible lack of regularity of its derivatives. To overcome this difficulty  and the additional complications due to the infinite dimensional setting, we use again the approximants in Theorem \ref{main theorem}.

\begin{thm}\label{chiavi1}
For $p\in (1,+\infty)$, $f\in L^p(\Omega,\nu)$ and $t>0$ there exists a positive constant $K_p$, depending only on $p$, such that
\begin{align}\label{stima_integrata}
|D_H T_\Omega(t)f|^p_H &\leq K_p t^{-\frac{p}{2}} T_\Omega(t)|f|^p,\qquad \nu\text{-a.e. in }\Omega.
\end{align}
As a consequence, we get
\begin{align}\label{stima_radice}
\norm{D_H T_\Omega(t)f}_{L^p(\Omega,\nu;H)}
&\leq K_p^{\frac{1}{p}} t^{-\frac{1}{2}} \norm{f}_{L^{p}(\Omega,\nu)}.
\end{align}
\end{thm}
\begin{proof}
We remark that \eqref{stima_radice} is an easy consequence of \eqref{stima_integrata}, so it is enough to prove \eqref{stima_integrata}. We divide the proof in two steps. In the first step we prove that if $f\in\mathcal{F}C_b^\infty(X)$, then for every $\eps,s>0$ and $p\in(1,2]$ there exists $K_p>0$, depending only on $p$, such that
\begin{align}\label{stima_radice_intermedia}
|D_H T_\eps(s)f_n|^p_H &\leq K_p s^{-\frac{p}{2}} T_\eps(s)|f_n|^p,
\qquad \nu_\eps\text{-a.e. in } X,
\end{align}
(see Theorem \ref{main theorem}). In the second step we prove \eqref{stima_integrata} for any $p\in(1,\infty)$ and $f\in L^p(\Omega,\nu)$.

\noindent\emph{Step 1.}
Let us differentiate the function
\begin{gather*}
G_{\delta,n}(t)=T_\eps(s-t)\pa{\pa{|T_\eps(t)f_n|^2+\delta}^{p/2}-\delta^{p/2}},\qquad 0< t< s,
\end{gather*}
where $\eps,\delta>0$ and $p\in(1,2]$. Setting $\phi_{\eps,\delta,n}(t):=|T_\eps(t)f_n|^2+\delta$, we have
\begin{align}
G_{\delta,n}'(t)=&\, -L_\eps T_\eps(s-t)\pa{\pa{\phi_{\eps,\delta,n}(t)}^{p/2}-\delta^{p/2}}
\notag \\
&+T_\eps(s-t)\pa{p\pa{\phi_{\eps,\delta,n}(t)}^{(p-2)/2}(T_\eps(t)f_n) (L_\eps T_\eps(t)f_n)}
\notag\\
=&\, T_\eps(s-t)\Big[-L_\eps\pa{\pa{\phi_{\eps,\delta,n}(t)}^{p/2}-\delta^{p/2}}
\notag \\
&+p(T_\eps(t)f_n) (L_\eps T_\eps(t)f_n)\pa{\phi_{\eps,\delta,n}(t)}^{(p-2)/2}\Bigr].
\label{sandra} \end{align}
By Theorem \ref{main theorem} the function $\pa{\phi_{\eps,\delta,n}(t)}^{p/2}-\delta^{p/2}$ belongs to
$\mathcal{F}C^3_b(X)$, hence from the definition of $L_\ve$ (see \eqref{operator_veps}) we get
\begin{align}\label{qualc}
L_\eps\pa{\pa{\phi_{\eps,\delta,n}(t)}^{p/2}-\delta^{p/2}}=&\, p\pa{\phi_{\eps,\delta,n}(t)}^{(p-2)/2}(T_\eps(t)f_n)(L_\eps T_\eps(t)f_n)\notag\\
&+p\pa{\phi_{\eps,\delta,n}(t)}^{(p-2)/2}|D_H T_\eps(t)f_n|_H^2\notag\\
&+p(p-2)\pa{\phi_{\eps,\delta,n}(t)}^{(p-4)/2}(T_\eps(t) f_n)^2|D_H T_\eps(t)f_n|_H^2.
\end{align}
Combining \eqref{sandra} and \eqref{qualc} we get
\begin{align*}
G_{\delta,n}'(t)=\, &-p T_\eps(s-t)\pa{\pa{\phi_{\eps,\delta,n}(t)}^{(p-2)/2}|D_H T_\eps(t)f_n|_H^2}\\
&+p(2-p)T_\eps(s-t)\pa{\pa{\phi_{\eps,\delta,n}(t)}^{(p-4)/2}(T_\eps(t) f_n)^2|D_H T_\eps(t)f_n|_H^2}.
\end{align*}
Since the semigroup $(T_\eps(t))_{t\geq 0}$ is positivity preserving (see Proposition \ref{prop_TOm}(ii) and Remark \ref{remark_Teps}) we get
\begin{align}\label{tastiera}
G_{\delta,n}'(t)&\leq p(1-p) T_\eps(s-t)\pa{\pa{\phi_{\eps,\delta,n}(t)}^{(p-2)/2}|D_H T_\eps(t)f_n|_H^2}.
\end{align}
Now integrating \eqref{tastiera} from $0$ to $s$ with respect to $t$, we get
\begin{align*}
T_\eps(s)\pa{\pa{|f_n|^2+\delta}^{p/2}-\delta^{p/2}}&-\pa{|T_\eps(s)f_n|^2+\delta}^{p/2}+\delta^{p/2}
\\
&\leq p(1-p)\int_0^s T_\eps(s-t)\pa{\pa{|T_\eps(t) f_n|^2+\delta}^{(p-2)/2}|D_H T_\eps(t)f_n|_H^2} dt.
\end{align*}
Using again that $(T_\eps(t))_{t\geq 0}$ is positivity preserving, from the previous inequality we get
\begin{gather}\label{formula_intermedia}
p(p-1)\int_0^s T_\eps(s-t)\pa{\pa{|T_\eps(t) f_n|^2+\delta}^{(p-2)/2}|D_H T_\eps(t)f_n|_H^2} dt\leq \pa{|T_\eps(s) f_n|^2+\delta}^{p/2}.
\end{gather}
By the semigroup property, \eqref{point}, \eqref{jensen}, \eqref{holder}, Remark \ref{remark_Teps} and the Young inequality, we get for every $\eta>0$
\begin{align}
|D_H T_\eps(s)f_n|^p_H =&\, |D_H T_\eps(s-t)T_\eps(t)f_n|_H^p
\notag\\
 \leq&\,  e^{-p\lambda_1^{-1}(s-t)}T_\eps(s-t)|D_H T_\eps(t)f_n|_H^p
\notag\\
 \leq&\,  e^{-p\lambda_1^{-1}(s-t)}T_\eps(s-t)\pa{\pa{\phi_{\eps,\delta,n}(t)}^{-\frac{p(2-p)}{4}}
 |D_H T_\eps(t)f_n|_H^p\pa{\phi_{\eps,\delta,n}(t)}^{\frac{p(2-p)}{4}}}
\notag\\
 \leq&\,  e^{-p\lambda_1^{-1}(s-t)}\pa{T_\eps(s-t)\pa{\pa{\phi_{\eps,\delta,n}(t)}^{\frac{p}{2}-1}
 |D_H T_\eps(t)f_n|_H^2}}^{p/2}
\notag\\
& {\hskip 2cm} \cdot \pa{T_\eps(s-t)\pa{\phi_{\eps,\delta,n}(t)}^{\frac{p}{2}}}^{1-\frac{p}{2}}
\notag\\
 \leq&\,   e^{-p\lambda_1^{-1}(s-t)}\frac{p}{2}\eta^{2/p}T_\eps(s-t)\pa{\pa{\phi_{\eps,\delta,n}(t)}^{\frac{p}{2}-1}|D_H T_\eps(t)f_n|_H^2}
\notag\\
& + e^{-p\lambda_1^{-1}(s-t)}\pa{1-\frac{p}{2}}\eta^{2/(p-2)}T_\eps(s-t)\pa{|T_\eps(t)f_n|^p+\delta^{p/2}}\notag\\
 \leq&\,   e^{-p\lambda_1^{-1}(s-t)}\frac{p}{2}\eta^{2/p}T_\eps(s-t)\pa{\pa{\phi_{\eps,\delta,n}(t)}^{\frac{p}{2}-1}|D_H T_\eps(t)f_n|_H^2}
\notag\\
& + e^{-p\lambda_1^{-1}(s-t)}\pa{1-\frac{p}{2}}\eta^{2/(p-2)}T_\eps(s-t)\pa{T_\eps(t)|f_n|^p+\delta^{p/2}}\label{stima_intermedia_gradiente}
\end{align}
Multiplying \eqref{stima_intermedia_gradiente} by $e^{p\lambda_1^{-1}(s-t)}$, integrating from $0$ to $s$ with respect to $t$, and recalling \eqref{formula_intermedia} we get
\begin{align*}
\frac{e^{p\lambda_1^{-1}(s-t)}-1}{p\lambda_1^{-1}}|D_H T_\eps(s)f_n|^p_H \leq&\,
\frac{\eta^{2/p}}{2(p-1)}\pa{|T_\eps(s) f_n|^2+\delta}^{p/2}
\\
& +\pa{1-\frac{p}{2}}\eta^{2/(p-2)}s\pa{T_\eps(s)|f_n|^p+\delta^{p/2}}.
\end{align*}
Letting $\delta\ra 0^+$ and applying \eqref{jensen} we obtain
\begin{align*}
\frac{e^{p\lambda_1^{-1}(s-t)}-1}{p\lambda_1^{-1}}|D_H T_\eps(s)f_n|^p_H &\leq \pa{\frac{\eta^{2/p}}{2(p-1)}+\pa{1-\frac{p}{2}}\eta^{2/(p-2)}s} T_\eps(s)|f_n|^p
\end{align*}
whence
\begin{align*}
\frac{e^{p\lambda_1^{-1}(s-t)}-1}{p\lambda_1^{-1}}|D_H T_\eps(s)f_n|^p_H &\leq \min_{\eta>0}\set{\frac{\eta^{2/p}}{2(p-1)}+\pa{1-\frac{p}{2}}\eta^{2/(p-2)}s} T_\eps(s)|f_n|^p\\&=:c_p s^{1-\frac{p}{2}} T_\eps(s)|f_n|^p.
\end{align*}
for some positive constant $c_p$ depending only on $p$. Setting $t=0$, and recalling that the function $s/(e^{p\lambda_1^{-1}s}-1)$ is bounded from above, we get \eqref{stima_radice_intermedia}.

\noindent\emph{Step 2.} If $p\in(2,\infty)$ it suffices to write $|D_H T_\eps(s)f_n|^p_H=(|D_H T_\eps(s)f_n|^2_H)^{p/2}$ and to apply \eqref{stima_radice_intermedia} with $p=2$. Then, using \eqref{holder} together with Remark \ref{remark_Teps}, we get \eqref{stima_radice_intermedia} for every $p\in(1,\infty)$. Due to the properties listed in Theorem \ref{main theorem}, letting $n\ra+\infty$ and $\eps\ra 0$, uptoa subsequence we get \eqref{stima_integrata} for every $f\in\mathcal{F}C_b^{\infty}(X)$. Moreover,
 integrating it on $\Omega$ and using that $\nu$ is the invariant measure associated with $T_{\Om}(t)$, we get
\begin{equation}\label{c0-c1}
\int_\Omega |D_H T_\Omega(s)f|^p_Hd\nu \leq K_p s^{-\frac{p}{2}} \int_\Omega|f|^pd\nu.
\end{equation}
for any $f\in\mathcal{F}C^\infty_b(\Omega)$ and $p\in(1,\infty)$. Finally we extend estimate \eqref{c0-c1} to any $f\in L^p(\Om, \nu)$ arguing by approximation as in the last part of the proof of Theorem \ref{pointwise}. To this aim, let $f\in L^p(\Om,\nu)$ and let $(g_n)_{n\in\N}$ be a sequence of functions in $\mathcal{F}C^\infty_b(\Omega)$ converging to $f$ in $L^p(\Omega,\nu)$. Then, for every $n,k\in\N$
\begin{align*}
\int_\Omega |D_H T_\Omega(s)g_n-D_H T_\Omega(s)g_k|^p_Hd\nu &\leq K_p s^{-\frac{p}{2}} \int_\Omega|g_n-g_k|^pd\nu.
\end{align*}
So the sequence $(D_H T_\Omega(s)g_n)_{n\in\N}$ is a Cauchy sequence in $L^p(\Omega,\nu;H)$. The closability of the operator $D_H:\mathcal{F}C_b^\infty(\Omega)\ra L^p(\Omega,\nu)$ in $L^p(\Om,\nu)$ and the fact that for any $s>0$ the sequence $(T_\Omega(s)g_n)_{n\in\N}$ converges to $T_\Omega(s)f$ we get that $\lim_{n\ra+\infty}D_H T_\Omega(s)g_n=D_H T_\Omega(s)f$ in $L^p(\Omega,\nu;H)$. Hence, writing \eqref{c0-c1} with $f$ replaced by $g_n$ and letting $n\to+\infty$, we conclude.
\end{proof}

The pointwise gradient estimate \eqref{point} implies that $\norm{D_H T_\Omega(t)f}_{L^p(\Omega,\nu;H)}$ vanishes as $t\to+\infty$ and $f\in D^{1,p}(\Omega,\nu)$. Actually using \eqref{stima_integrata} we get the same result when $f$ belongs to $L^p(\Omega,\nu)$.

\begin{cor}\label{esp-decay-grad}
Let $p\in(1,\infty)$ and $t\ge 1$. For every $f\in L^p(\Omega,\nu)$
\begin{align*}
\norm{D_H T_\Omega(t)f}_{L^p(\Omega,\nu;H)} &\leq  C_p e^{-\lambda_1^{-1}t} \norm{f}_{L^p(\Omega,\nu)},
\end{align*}
where $C_p=K_p^{1/p}e^{\lambda_1^{-1}}$ and $K_p$ is the positive constant in Theorem \ref{chiavi1}.
\end{cor}

\begin{proof}
By \eqref{point}, \eqref{stima_radice}, the semigroup property and the fact that $\nu$ is invariant with respect to $T_\Om(t)$ we get
\begin{align*}
\int_\Omega|D_H T_\Omega(t)f|_H^pd\nu&=\int_\Omega|D_H T_\Omega(t-1)T_\Omega(1)f|_H^pd\nu
\\
&\leq e^{-p\lambda_1^{-1}(t-1)}\int_\Omega T_\Omega(t-1)|D_H T_\Omega(1)f|_H^pd\nu
\\
&\leq e^{-p\lambda_1^{-1}(t-1)}\int_\Omega |D_H T_\Omega(1)f|_H^pd\nu\\
&\leq K_p e^{-p\lambda_1^{-1}(t-1)}\int_\Omega T_\Omega(1)|f|^pd\nu\\
&\leq K_pe^{-p\lambda_1^{-1}(t-1)}\int_\Omega |f|^pd\nu,
\end{align*}
for any $t\geq 1,\, f\in L^p(\Om,\nu)$. This concludes the proof.
\end{proof}

\section{Logarithmic Sobolev inequality and other consequences}\label{SecLogSob}
Logarithmic Sobolev inequalities are important tools in the study of Gaussian Sobolev spaces since they represent the counterpart of the Sobolev embeddings which in general fail to hold when the Lebesgue measure is replaced by other measures, as for example the Gaussian one. In infinite dimension such inequalities are known for the Gaussian measure on the whole space (see \cite[Theorem 5.5.1]{Bog}) and on convex domains (see \cite[Proposition 3.5]{Cap15}). In the weighted Gaussian case the inequality is known in the whole space (see \cite[Proposition 11.2.19]{DaP02}), for Fr\'echet differentiable functions. In this section we use the pointwise gradient estimates \eqref{point} and \eqref{stima_integrata} to prove logarithmic Sobolev inequalities for weighted Gaussian measures on convex domains generalising all the above results. We also collect some consequences of the logarithmic Sobolev inequality \eqref{logsob}. To simplify the notation we set, if$f\in L^1(X,\nu_\varepsilon)$ and $g\in L^1(X,\nu)$
\begin{equation}\label{average}
m_\eps(f):=\frac{1}{\nu_\eps(X)}\int_X fd\nu_\eps,\qquad
m_\Omega(g):=\frac{1}{\nu(\Omega)}\int_\Omega gd\nu.
\end{equation}

First of all we study the asymptotic behaviour of the semigroup $(T_\eps(t))_{t\geq 0}$.

\begin{lemm}\label{nero}
For any $\eps>0$ and  $f\in\mathcal{F}C^1_b(X)$
\begin{gather}\label{rosso}
\lim_{t\ra+\infty}T_\eps(t)f(x)=m_\eps(f),\qquad\nu_\eps\text{\rm-a.e. }x\in X.
\end{gather}
In addition, if $f\leq 1$ and has a positive infimum, then
\begin{equation}\label{blu}
\lim_{t\ra+\infty}\int_X (T_\eps(t) f)\log (T_\eps(t) f)d\nu_\eps=\left(\int_X f d\nu_\eps\right)\log \pa{m_\eps(f)}=
\nu_\eps(X)m_\eps(f)\log \pa{m_\eps(f)}.
\end{equation}
\end{lemm}
\begin{proof}
First of all note that since the function $(0,1]\ni x\mapsto x\,|\log x|$ has a maximum, formula \eqref{blu} can be obtained by \eqref{rosso} and the dominated convergence theorem.
The proof of \eqref{rosso} is divided in three steps.

\noindent\emph{Step 1.} Let us show that there exists a sequence $(t_k)_{k\in\N}\subseteq [0,+\infty)$, such that $t_k\ra +\infty$ as $k\to +\infty$ and $T_\eps(t_k)f\to g_\eps$ weakly in $L^2(X,\nu_\eps)$ for some $g_\eps\in L^2(X,\nu_\eps)$, as $k$ goes to infinity. To do this, it is sufficient to consider a sequence $(t_n)_{n\in\N}$ tending to $+\infty$ as $n\to +\infty$ and to recall that $T_\eps(t_n)$ is a contraction in $L^2(X,\nu_\ve)$.

\noindent\emph{Step 2.} Here we claim that $g_\eps$ is $H$-invariant, i.e., $g_\eps(x+h)=g_\eps(x)$ for $\gamma$-a.e. $x\in X$ and for every $h\in H$.
For any $\varphi\in C_b(X)$ we have
\begin{align*}
\left|\int_X \big[g_\ve(x+h)-g_\ve(x)\big]\varphi(x)d\nu_\ve(x)\right|
&\leq\,  \left|\int_X \big[g_\ve(x+h)-(T_{\ve}(t_k)f)(x+h)\big]\varphi(x)d\nu_\ve(x)\right|\tag{$I_1$}
\\
&+ \left|\int_X \big[(T_{\ve}(t_k)f)(x+h)-(T_{\ve}(t_k)f_n)(x+h)\big]\varphi(x)d\nu_\ve(x)\right|\tag{$I_2$}\\
&+\left|\int_X \big[(T_{\ve}(t_k)f_n)(x+h)-(T_{\ve}(t_k)f_n)(x)\big]\varphi(x)d\nu_\ve(x)\right|\tag{$I_3$}\\
&+\left|\int_X \big[(T_{\ve}(t_k)f_n)(x)-(T_{\ve}(t_k)f)(x)\big]\varphi(x)d\nu_\ve(x)\right|
\tag{$I_4$}\\
&+\left|\int_X \big[(T_{\ve}(t_k)f)(x)-g_\ve(x)\big]\varphi(x)d\nu_\ve(x)\right|
\tag{$I_5$}
\end{align*}
where $(f_n)_{n\in\N}$ is the sequence in Theorem \ref{main theorem}.
The regularity of $T_{\ve}(t_k)f_n$ and \eqref{biglietto} allow us to estimate $(I_3)$ as follows
\begin{align*}
(I_3)&=\abs{\int_X \pa{\int_0^1\gen{D_H T_\eps(t_k)f_n(x+sh),h}_Hds}\varphi(x)d\nu_\eps(x)}
\\
& \leq  e^{-\lambda^{-1}_1t_k}\abs{h}_H\norm{\varphi}_\infty
\int_0^1 \int_X(T_\eps(t_k)\abs{D_H f_n}_H)(x+sh)d\nu_\eps(x)ds
\\
&\leq  e^{-\lambda^{-1}_1t_k}\abs{h}_H\norm{\varphi}_\infty
\int_0^1 \int_X\abs{D_H f_n(x+sh)}_Hd\nu_\eps(x)ds
\\
&\leq  e^{-\lambda^{-1}_1t_k}\abs{h}_H\norm{\varphi}_
\infty\Bigl(\int_0^1 \int_X\abs{D_H f(x+sh)}_Hd\nu_\eps(x)ds+M\Bigr)
\\
&\leq e^{-\lambda^{-1}_1t_k}\abs{h}_H\norm{\varphi}_\infty(\nu_\eps(X)\norm{D_H f}_\infty+M),
\end{align*}
for some positive $M$, where in the second to last line we took into account that $\|D_H f_n\|_{L^1(X,\nu_\eps;H)}$ converges to $\|D_H f\|_{L^1(X,\nu_\eps;H)}$ as $n\to +\infty$. Now, for every $\eta>0$ we can choose $k$ large enough such that $(I_1)+(I_3)+(I_5)\le \eta/2$ and $n$ such that $(I_2)+(I_4)\le \eta/2$. This proves the claim.

\noindent\emph{Step 3.} In this step we complete the proof. By \cite[Theorem 2.5.2]{Bog} a $H$-invariant function coincides $\gamma$-a.e. in $X$ (hence $\nu$-a.e. in $X$ as well) with a constant function, i.e., there exists $c\in\R$ such that $g_\eps(x)=c$ for $\gamma$-a.e. $x\in X$. We get
\begin{gather*}
c=\frac{1}{\nu_\eps(X)}\int_Xcd\nu_\eps=\frac{1}{\nu_\eps(X)}\lim_{k\ra+\infty}\int_X T_\eps(t_k)fd\nu_\eps=m_\eps(f)
\end{gather*}
where in the last equality we used the invariance of $\nu_\eps$ with respect to $T_\eps(t)$.
Since our arguments are independent of the sequence $(t_k)_{k\in\N}$, we get \eqref{rosso}.
\end{proof}

\begin{rmk}
{\rm In view of the method used in the proof, the results in Lemma \ref{nero} cannot be easily extended to the semigroup $T_\Omega(t)$. However, as we prove in Proposition \ref{cavo}, the asymptotic behaviour of $T_\Om(t)$ as $t\ra+\infty$ can be obtained also with a precise decay estimate.}
\end{rmk}

Now we are ready to prove that the measure $\nu$ satisfies a logarithmic Sobolev inequality in $\Om$. The idea of the proof is to apply the Deuschel and Stroock method (see \cite{DS90}) to the measure $\nu_\eps$ and then taking the limit as $\eps\to 0$.
\begin{prop}\label{pro_logsob}
For $p\in [1,\infty)$ and $f\in \mathcal{F}C^1_b(\Omega)$, the following inequality holds:
\begin{align}\label{logsob}
\int_\Omega\abs{f}^p\log\abs{f}^pd\nu\leq\nu(\Om)m_{\Om}(|f|^p)&\log\pa{m_\Omega (\abs{f}^p)}+\frac{p^2\lambda_1}{2}\int_\Omega\abs{f}^{p-2}\abs{D_H f}_H^2\chi_{\set{f\neq 0}}d\nu.
\end{align}
\end{prop}

\begin{proof}
We split the proof in two parts. In the first part we prove the claim when $f$ satisfies some additional hypotheses and in the second part we show \eqref{logsob} in its full generality.

\noindent\emph{Step 1.} Here we prove \eqref{logsob} with $\nu$ and $\Om$ replaced by $\nu_\eps$ and $X$, and $f$ in $\mathcal{F}C^1_b(X)$ such that there exists a positive constant $c$ with $c\leq f\leq 1$. To this aim we consider the function
\begin{equation*}
F_\eps(t)=\int_X (T_\eps(t)f^p)\log (T_\eps(t)f^p)d\nu_\eps,\qquad t\geq 0.
\end{equation*}
which is well defined thanks to Proposition \ref{prop_TOm}(ii)-(iii) and Remark \ref{remark_Teps}.

 Our aim is to find a bound from below for the derivative of $F_\eps$. Indeed, we show that $F_\eps'\geq c_{1} e^{-c_{2}t}\int_X f^{p-2}\abs{D_Hf}_H^2d\nu_\eps$, for some positive constants $c_{1}$ and $c_{2}$. We start by observing that
\begin{align*}
F'_\eps(t)&=\int_X (L_\eps T_\eps(t)f^p)\log (T_\eps(t)f^p)d\nu_\eps+\int_X L_\eps T_\eps(t)f^pd\nu_\eps\\
&=-\int_X\gen{D_H T_\eps(t)f^p,D_H \log (T_\eps(t)f^p)}_Hd\nu_\eps\\
&=-\int_X\frac{\abs{D_H T_\eps(t)f^p}^2_H}{T_\eps(t)f^p}d\nu_\eps
\end{align*}
where we used that $\int_X L_\eps \varphi d\nu_\eps=0$ for any $\varphi \in D(L_{\ve})$, the definition of $L_\eps$ and the integration by parts formula.
By \eqref{holder} and Remark \ref{remark_Teps} we have $T_\eps(t)\abs{D_Hf^p}_H\leq \pa{T_\eps(t)\frac{\abs{D_Hf^p}^2_H}{f^p}}^{1/2}\pa{T_\eps (t)f^p}^{1/2}$. Hence, by using \eqref{point} we deduce
\begin{align*}
F'_\eps(t)&\geq -e^{-2\lambda_1^{-1}t}
\int_X\frac{(T_\eps(t)\abs{D_Hf^p}_H)^2}{T_\eps(t)f^p}d\nu_\eps
\geq -e^{-2\lambda_1^{-1}t}\int_X T_\eps(t)\left(\frac{\abs{D_Hf^p}^2_H}{f^p}\right)d\nu_\eps
\\
&=-e^{-2\lambda_1^{-1}t}p^2\int_X f^{p-2}\abs{D_Hf}_H^2d\nu_\eps.
\end{align*}
Integrating from $0$ to $+\infty$ and using \eqref{blu} we get
\begin{gather*}
\int_Xf^p\log f^pd\nu_\eps\leq \pa{\int_Xf^pd\nu_\eps}\log\pa{m_\eps(f^p)}+\frac{p^2\lambda_1}{2}\int_Xf^{p-2}\abs{D_H f}_H^2d\nu_\eps.
\end{gather*}
Finally letting $\ve\to 0$ and recalling that $\nu_\eps$ weakly$^*$ converges to $\chi_\Om\nu$, we get the claim.

\noindent\emph{Step 2.} Now, for any $f\in\mathcal{F}C^1_b(\Omega)$ and $n\in \N$ let consider the sequence $(f_n)_{n \in \N}$ defined by $f_n=(1+\norm{f}_\infty)^{-1}\sqrt{f^2+n^{-1}}$. Step 1 yields that
\begin{gather}\label{cable}
\int_\Omega f_n^p\log(f_n^p)d\nu\leq \pa{\int_\Omega f_n^pd\nu}\log\pa{m_\Omega (f_n^p)}+\frac{p^2\lambda_1}{2}\int_\Omega f_n^{p-2}\abs{D_H f_n}_H^2d\nu.
\end{gather}
Observing that there exists a positive constant $c_{n,p}$ such that $c_{n,p} \le f_n^p\leq 1$ for any $n\in \N$ and using the fact that the function $x\mapsto x\abs{\log x}$ is bounded in $(0,1]$, by the dominated convergence theorem the left hand side of \eqref{cable} converges to
\[
(1+\norm{f}_\infty)^{-p}\int_\Omega\abs{f}^p\log\big[(1+\norm{f}_\infty)^{-p}\abs{f}^p\big]d\nu,
\]
and the first term in the right hand side of \eqref{cable} converges to
\[
\pa{(1+\norm{f}_\infty)^{-p}\int_\Omega\abs{f}^pd\nu}\log\pa{\frac{m_\Omega(\abs{f}^p)}{(1+\norm{f}_\infty)^{p}}}.
\]
Since $|D_H f_n|_H\leq (1+\norm{f}_\infty)^{-1}|D_H f|_H$ for every $n\in\N$, by the monotone convergence theorem if $p\in[1,2)$, and by Lebesgue's dominated convergence theorem otherwise, we obtain
\begin{gather*}
\lim_{n\ra+\infty}\int_\Omega{f_n}^{p-2}\abs{D_H f_n}_H^2d\nu=
(1+\norm{f}_\infty)^{-p} \int_\Omega\abs{f}^{p-2}\abs{D_H f}_H^2\chi_{\set{f\neq 0}}d\nu.
\end{gather*}
So the statement follows letting $n$ to infinity in \eqref{cable}.
\end{proof}

As it is well known the logarithmic Sobolev inequality has several interesting consequences. Among them, we point out the following, related to our setting: once a log-Sobolev inequality with respect to the measure $\nu$ has been proved, a summability improving property of $T_\Om(t)$ follows. Indeed we are able to show that $T_{\Om}(t)$ maps $L^q(\Om, \nu)$ into $L^p(\Om, \nu)$ for some $p>q$. The technique used to prove this property is quite standard. However, for the sake of completeness, we provide a proof of it.

\begin{prop}
Let $t>0$ and $p,q\in(1,+\infty)$ be such that $p\leq (q-1)e^{2\lambda_1^{-1} t}+1$.
Then the operator $T_\Omega(t)$ maps $L^q(\Omega,\nu)$ in $L^p(\Omega,\nu)$ and
\begin{gather}\label{hyper}
\norm{T_{\Omega}(t)f}_{L^p\pa{\Omega,\nu}}\leq [\nu(\Om)]^{\frac{1}{p}-\frac{1}{q}}
\norm{f}_{L^q\pa{\Omega,\nu}},\qquad t>0,\ f\in L^q(\Omega,\nu).
\end{gather}
\end{prop}
\begin{proof}
Let $f\in\mathcal{F}C^1_b(\Omega)$, with a positive global infimum, and let $p(t):=(q-1)e^{2\lambda_1^{-1} t}+1$. For $s\geq 0$ we set
\[
G(s):=\pa{\frac{1}{\nu(\Omega)}\int_\Omega(T_\Omega(s)f)^{p(s)}d\nu}^{1/p(s)}=:\pa{\frac{1}{\nu(\Omega)}F(s)}^{1/p(s)}
\]
and we prove that $G(s)$ is a non-increasing function. Before starting we want to recall that $T_\Omega(s)$ maps $\mathcal{F}C^1_b(\Omega)$ into $D^{1,2}(\Omega,\nu)\cap L^\infty(\Omega,\nu)$, due to the  definition of the operator $T_\Omega(s)$ and Proposition \ref{prop_TOm}(ii). This guarantees that all the integrals we are going to write are well defined and finite. So, using \eqref{defn_LOm}, we get
\begin{align}\label{333}
F'(s)&=p'(s)\int_\Omega(T_\Omega(s)f)^{p(s)}\log(T_\Omega(s)f)d\nu-p(s)(p(s)-1)
\int_\Omega (T_\Omega(s)f)^{p(s)-2}\abs{D_HT_\Omega(s)f}_H^2d\nu.
\end{align}
Now we set $u(s):=T_\Omega(s)f$ and we differentiate the function $G$. Taking into account \eqref{333}, we get
\begin{gather*}
G'=\, G\Biggl(-\frac{p'}{p^2}\log (m_\Omega(u^p))+\frac{1}{p\int_\Omega u^{p}d\nu}\pa{p'\int_\Omega u^{p}\log ud\nu-p(p-1)\int_\Omega u^{p-2}\abs{D_Hu}_H^2d\nu}\Biggr)
\\
=\, G\frac{p'}{p^2\int_\Omega u^{p}d\nu}\pa{-\pa{\int_\Omega u^{p}d\nu}\log \pa{m_\Omega(u^{p})}+\int_\Omega u^{p}\log u^pd\nu}-\frac{G(p-1)}{\int_\Omega u^{p}d\nu}\int_\Omega u^{p-2}\abs{D_Hu}_H^2d\nu.
\end{gather*}
Since $p'(s)=2\lambda_1^{-1}(q-1)e^{2\lambda_1^{-1}s}\geq0$ we can apply \eqref{logsob} to get
\begin{gather*}
G'(s)\leq (G(s))^{1-p(s)}\pa{\frac{p'(s)\lambda_1}{2}-(p(s)-1)}\int_\Omega (T_\Omega(s)f)^{p(s)-2}\abs{D_HT_\Omega(s)f}_H^2d\nu=0.
\end{gather*}
This proves that $G$ is a decreasing function, which means that $G(0)\geq G(t)$ for every $t>0$, i.e.,
\begin{gather*}
\norm{T_{\Omega}(t)f}_{L^{p(t)}\pa{\Omega,\nu}}\leq [\nu(\Omega)]^{\frac{1}{p(t)}-\frac{1}{q}}\norm{f}_{L^q\pa{\Omega,\nu}}.
\end{gather*}
So we get \eqref{hyper} for a function $f\in \mathcal{F}C^1_b(\Omega)$ with positive global infimum. Indeed, if $p<p(t)$
\begin{align*}
\norm{T_\Omega(t)f}_{L^p(\Omega,\nu)}&
\leq [\nu(\Omega)]^{\frac{p(t)-p}{p(t)p}}\norm{T_\Omega(t)f}_{L^{p(t)}(\Omega,\nu)}
\\
&\leq [\nu(\Omega)]^{\frac{p(t)-p}{p(t)p}}[\nu(\Omega)]^{\frac{1}{p(t)}
-\frac{1}{q}}\norm{f}_{L^{q}(\Omega,\nu)}= [\nu(\Omega)]^{\frac{1}{p}-\frac{1}{q}}\norm{f}_{L^{q}(\Omega,\nu)}.
\end{align*}
Arguing as in the second step of the proof of Proposition \ref{pro_logsob} we obtain \eqref{hyper} for a general $f\in \mathcal{F}C^1_b(\Omega)$. The density of the space $\mathcal{F}C^1_b(\Omega)$ in $L^q\pa{\Omega,\nu}$ allows us to conclude the proof.
\end{proof}

From the logarithmic Sobolev inequality follows the asymptotic behaviour of $T_\Omega(t)f$ as $t$ goes to infinity, whenever $f$ belongs to $L^2(\Om, \nu)$. This can be done thanks to the Poincar\'e inequality.

\begin{prop}
Let $p\in[2,\infty)$ and $f\in D^{1,p}(\Omega,\nu)$. Then
\begin{gather}\label{poin}
\norm{f-m_\Omega(f)}_{L^p(\Omega,\nu)}\leq K\norm{D_H f}_{L^p(\Omega,\nu;H)},
\end{gather}
where $K$ is a positive constant depending only on $p$, $\lambda_1$ and $\nu(\Omega)$. Furthermore if $p=2$ then $K=\lambda_1^{1/2}$.
\end{prop}

\begin{proof}
We divide the proof in two steps. In the first step we prove \eqref{poin} for $p=2$, while in the second step we prove the claim for $p\in(2,\infty)$.

\noindent \emph{Step 1.} We use an idea of \cite{Rot81} (see also \cite[Theorem 5.2]{ALL13}). Let $f\in \mathcal{F}C^1_b(\Omega)$, $\eta>0$ and consider the function $f_\eta=1+\eta(f-m_\Omega(f))$. Recalling that $(1+\xi)^2\log(1+\xi)^2=2\xi+3\xi^2+o(\xi^2)$ as $\xi\to 0$, we get
\begin{gather*}
\int_\Omega f_\eta^2\log f_\eta^2d\nu-\pa{\int_\Omega f^2_\eta d\nu}\log\pa{m_\Omega(f^2_\eta)}=2\eta^2\int_\Omega \pa{f-m_\Omega(f)}^2d\nu+o(\eta^2).
\end{gather*}
By \eqref{logsob}, with $p=2$ and $f$ replaced by $f_\eta$, we get
\begin{gather*}
2\eta^2 \int_\Omega \pa{f-m_\Omega(f)}^2d\nu+o(\eta^2)\leq 2\lambda_1\int_\Omega \abs{D_H f_\eta}^2_Hd\nu=2\lambda_1\eta^2\int_\Omega \abs{D_H f}^2_Hd\nu.
\end{gather*}
Letting $\eta\ra 0^+$ we get \eqref{poin} for a function $f$ belonging to $\mathcal{F}C^1_b(\Omega)$. Then by the density of $\mathcal{F}C^1_b(\Omega)$ in $D^{1,2}(\Omega,\nu)$ we get
\begin{gather}\label{poin2}
\int_\Omega \pa{f-m_\Omega(f)}^2d\nu\leq \lambda_1\int_\Omega \abs{D_H f}^2_Hd\nu,\qquad f\in D^{1,2}(\Omega,\nu).
\end{gather}

\noindent \emph{Step 2.} Now let assume that $p\in(2,\infty)$. If $g\in D^{1,p}(\Omega,\nu)$, then $|g|^{p/2}\in D^{1,2}(\Omega,\nu)$. This can be seen by approximating $g$ by a sequence of functions in $\mathcal{F}C^{1}_b(\Omega)$, which is dense in $D^{1,p}(\Omega,\nu)$. Applying \eqref{poin2}, with $f$ replaced by $|g|^{p/2}$, we get
\begin{gather}\label{xxx}
\int_\Omega |g|^p d\nu-\frac{1}{\nu(\Omega)}\pa{\int_\Omega |g|^{p/2}d\nu}^2\leq \frac{\lambda_1 p^2}{4}\int_\Omega |g|^{p-2}|D_H g|_H^2 d\nu
\end{gather}
Applying the Young inequality to the right hand side of \eqref{xxx}, for every $\eta>0$ we have
\begin{gather*}
\int_\Omega |g|^p d\nu\leq \frac{\lambda_1 p(p-2)\eta^{p/(p-2)}}{4}\int_\Omega |g|^{p}d\nu+\frac{\lambda_1 p}{2\eta^{p/2}}\int_\Omega |D_H g|_H^p d\nu+\frac{1}{\nu(\Omega)}\pa{\int_\Omega |g|^{p/2}d\nu}^2.
\end{gather*}
Choosing $\eta>0$ such that $\eta^{p/(p-2)}\leq 4/(\lambda_1 p(p-2))$ and $K(p,\eta):=1-(\lambda_1 p(p-2)\eta^{p/(p-2)})/4$ we deduce
\begin{gather}\label{macch}
K(p,\eta)\int_\Omega |g|^p d\nu\leq \frac{\lambda_1 p}{2\eta^{p/2}}\int_\Omega |D_H g|_H^p d\nu+\frac{1}{\nu(\Omega)}\pa{\int_\Omega |g|^{p/2}d\nu}^2.
\end{gather}
Now we proceed by induction. If $p\in(2,4)$, then
\[
\int_\Omega |g|^{p/2}d\nu\leq \pa{\int_\Omega |g|^2d\nu}^{p/4}[\nu(\Omega)]^{(4-p)/4}
\]
and so by \eqref{macch}, for every $p\in(2,4]$
\begin{gather*}
K(p,\eta)\int_\Omega |g|^p d\nu\leq \frac{\lambda_1 p}{2\eta^{p/2}}\int_\Omega |D_H g|_H^p d\nu+\frac{1}{[\nu(\Omega)]^{(2-p)/2}}\pa{\int_\Omega |g|^{2}d\nu}^{p/2}.
\end{gather*}
If we let $g=f-m_\Omega(f)$ for a function $f\in D^{1,p}(\Omega,\nu)$ we get
\begin{align*}
K(p,\eta)\int_\Omega \abs{f-m_\Omega(f)}^p d\nu\leq&\,
\frac{\lambda_1 p}{2\eta^{p/2}}\int_\Omega |D_H f|_H^p d\nu+\frac{1}{[\nu(\Omega)]^{(2-p)/2}}\pa{\int_\Omega \abs{f-m_\Omega(f)}^{2}d\nu}^{p/2}.
\end{align*}
By \eqref{poin2} we get
\begin{gather}\label{accccc}
K(p,\eta)\int_\Omega \abs{f-m_\Omega(f)}^p d\nu\leq \frac{\lambda_1 p}{2\eta^{p/2}}\int_\Omega |D_H f|_H^p d\nu+\frac{\lambda_1^{p/2}}{[\nu(\Omega)]^{(p^2-4)/(2p)}}\int_\Omega |D_H f|^{p}d\nu,
\end{gather}
which proves the statement when $p\in(2,4]$. Now let $p\in(4,8]$. For any $f\in D^{1,p}(\Omega,\nu)$ we apply \eqref{macch} to the function $g=f-m_\Omega(f)$, and since $p/2\in (2,4]$, we can use \eqref{accccc} with $p/2$ instead of $p$, to get the thesis for $p\in (4,8]$. Iterating the above procedure we conclude the proof.
\end{proof}

A standard consequence of the Poincar\'e inequality is the convergence of $T_\Omega(t)f$ to $m_\Omega(f)$ (see \eqref{average}) in $L^2(\Om, \nu)$, as the following exponential decay estimate shows.

\begin{cor}
If $f\in L^2(\Omega,\nu)$, then
\begin{equation}\label{decay-2}
\norm{T_\Omega(t)f-m_\Omega(f)}_{L^2(\Omega,\nu)}\leq e^{-\lambda_1^{-1}t}\norm{f}_{L^2(\Omega,\nu)}.
\end{equation}
As a consequence for every $f\in L^2(\Omega,\nu)$ it holds
\begin{gather*}
\lim_{t\ra+\infty} T_\Omega(t)f=m_\Omega(f),\qquad\nu\text{\rm-a.e. in } \Omega.
\end{gather*}
\end{cor}
\begin{proof}
Let $G(s)=\int_\Omega\pa{T_\Omega(s)f-m_\Omega(f)}^2d\nu$. Using \eqref{defn_LOm} and \eqref{poin} we get
\begin{align*}
G'(s)&=\frac{d}{ds}\int_\Omega \pa{T_\Omega(s)f-m_\Omega(f)}^2d\nu
=2\int_\Omega (T_\Omega(s)f)(L_\Omega T_\Omega(s)f)d\nu
\\
&=-2\int_\Omega\abs{D_H T_\Omega(s)f}^2_Hd\nu\leq -\frac{2}{\lambda_1}\int_\Omega \pa{T_\Omega(s)f-m_\Omega(T_\Omega(s)f)}^2d\nu
\\
&=-\frac{2}{\lambda_1}\int_\Omega \pa{T_\Omega(s)f-m_\Omega(f)}^2d\nu= -\frac{2}{\lambda_1} G(s).
\end{align*}
Thus $G(t)\leq e^{-2\lambda_1^{-1}t}G(0)$, which means
\begin{align*}
\int_\Omega\pa{T_\Omega(t)f-m_\Omega(f)}^2d\nu\leq&\,
e^{-2\lambda_1^{-1}t}\int_\Omega\pa{f-m_\Omega(f)}^2d\nu
\\
=&\, e^{-2\lambda_1^{-1}t}\sq{\int_\Omega f^2d\nu-2\frac{1}{\nu(\Omega)}\pa{\int_\Omega fd\nu}^2+\frac{1}{\nu(\Omega)}\pa{\int_\Omega fd\nu}^2}
\\
\leq&\, e^{-2\lambda_1^{-1}t}\int_\Omega f^2d\nu.
\end{align*}
This concludes the proof.
\end{proof}

Once the Poincar\'e inequality, with $p=2$, the gradient estimate \eqref{stima_radice} and a hypercontractivity type estimate like \eqref{hyper} are available we can establish a relationship between the asymptotic behaviour of $T_\Omega(t)f$ and that of $|D_HT_\Omega(t)f|_H$ as $t\to+\infty$, whenever $f\in L^p(\Om,\nu)$, $p\in(1,\infty)$. More precisely, arguing as in \cite[Theorem 5.3]{ALL13} we can prove the following result, that extends the decay estimate \eqref{decay-2} to any $p\in(1,\infty)$. We skip the proof due to its length and the fact that it does not present any substantial difference with the one contained in \cite[Theorem 5.3]{ALL13}

\begin{prop}\label{cavo}
For any $p\in(1,\infty)$, consider the sets
\begin{align*}
{\mathcal A}_p=\Big\{&\omega\in\R\,\Big|\,\exists M_{p,\omega}>0 \text{ s.t. }\norm{T_\Omega(t)f-m_\Omega(f)}_{L^p(\Omega,\nu)}\leq M_{p,\omega}e^{\omega t}\norm{f}_{L^p(\Omega,\nu)},
\\
&t>0,\ f\in L^p(\Omega,\nu)\Big\};
\\
\mathcal{B}_p=\Big\{&\omega\in\R\,\Big|\,\exists N_{p,\omega}>0 \text{ s.t. }\norm{D_HT_\Omega(t)f}_{L^p(\Omega,\nu;H)}\leq N_{p,\omega}e^{\omega t}\norm{f}_{L^p(\Omega,\nu)},
\\
&t>1,\ f\in L^p(\Omega,\nu)\Big\}.
\end{align*}
Then the sets $\mathcal{A}_p$ and $\mathcal{B}_p$ are independent of $p$ and they coincide. In particular, by Corollary \ref{esp-decay-grad}, for any $p\in(1,\infty)$ there exists a positive constant $K_{p,\lambda_1}$, depending only on $p$ and $\lambda_1$, such that for every $t>0$ and $f\in L^p(\Omega,\nu)$, the inequality
\begin{equation*}
\norm{T_\Omega(t)f-m_\Omega(f)}_{L^p(\Omega,\nu)}\leq K_{p,\lambda_1} e^{-\lambda_1^{-1}t}\norm{f}_{L^p(\Omega,\nu)}
\end{equation*}
holds. As a consequence, for every $p\in(1,\infty)$ and $f\in L^p(\Omega,\nu)$
\[
\lim_{t\ra+\infty} T_\Omega(t)f = m_\Omega(f),\qquad\nu\text{-a.e.\,\, in\ }\Omega.
\]
\end{prop}


\begin{thebibliography}{99}

\bibitem{ACF17}
D. Addona, G. Cappa, S. Ferrari,
\newblock{\it On the domain of elliptic operators defined in subsets of Wiener spaces},
eprint arXiv:1706.05260 (2017).

\bibitem{AMM}
D. Addona, G. Menegatti, M. Miranda Jr.,
\newblock{\it On integration by parts formula on open convex sets in Wiener spaces},
preprint (2018).

\bibitem{ALL13}
L. Angiuli, L. Lorenzi, A. Lunardi,
\newblock{\it Hypercontractivity and asymptotic behavior in nonautonomous {K}olmogorov equations},
Comm. Partial Differential Equations {\bf 38} (2013), 2049--2080.

\bibitem{BE85}
D. Bakry, M. \'Emery,
\newblock{\it Diffusions hypercontractives}
in ``S\'eminaire de probabilit\'es, {XIX}, 1983/84'', Lecture Notes in Math., vol 1123 (1985), Springer, Berlin, 177--206.

\bibitem{BDaPT1}
V. Barbu, G. Da Prato, L. Tubaro,
\newblock{\it Kolmogorov equation associated to the stochastic reflection problem on a smooth convex set of a Hilbert space},
Ann. Probab. {\bf 37} (2009), 1427--1458.

\bibitem{BDaPT2}
V. Barbu, G. Da Prato, L. Tubaro,
\newblock{\it Kolmogorov equation associated to the stochastic reflection problem on a smooth convex set of a Hilbert space II},  Ann. Inst. Henri Poincar\'e Probab. Stat. {\bf 47} (2011), 699--724.

\bibitem{BDaPT3}
V. Barbu, G. Da Prato, L. Tubaro,
\newblock{\it The stochastic reflection problem in Hilbert spaces},
Comm. Partial Differential Equations {\bf 37} (2012), 352--367


\bibitem{BauCom}
H.H. Bauschke, P.L. Combettes,
\newblock{Convex analysis and monotone operator theory in Hilbert spaces},
CMS Books in Mathematics/Ouvrages de Math\'ematiques de la SMC, Springer, New York, 2011.


\bibitem{BF04}
M. Bertoldi, S. Fornaro
\newblock{\it Gradient estimates in parabolic problems with unbounded coefficients}, Studia Math. {\bf 165} (2004), 221--254.

\bibitem{Bog}
V. I. Bogachev,
\newblock{Gaussian measures}, Mathematical Surveys and Monographs, vol. 62, American Mathematical Society, Providence, RI, 1998.

\bibitem{Cap15}
G. Cappa,
\newblock{\it On the Ornstein-Uhlenbeck operator in convex sets of Banach spaces},
eprint arXiv:1503.02836 (2015), To appear in Studia Math.

\bibitem{CF16}
G. Cappa, S. Ferrari,
\newblock{\it Maximal Sobolev regularity for solutions of elliptic equations in infinite dimensional Banach spaces endowed with a weighted Gaussian measure},  J. Differential Equations {\bf 261} (2016), 7099--7131.

\bibitem{CF18}
G. Cappa, S. Ferrari,
\newblock{\it Maximal Sobolev regularity for solutions of elliptic equations in Banach spaces endowed with a weighted Gaussian measure: the convex subset case}, J. Math. Anal. Appl. {\bf 458} (2018), 300--331.

\bibitem{Cer01}
S. Cerrai,
\newblock{Second order {PDE}'s in finite and infinite dimension}, Lecture Notes in Mathematics, vol. 1762, Springer-Verlag, Berlin, 2001.

\bibitem{Cra91}
M. Cranston,
\newblock{\it Gradient estimates on manifolds using coupling}, J. Funct. Anal. {\bf 99} (1991), 110--124.

\bibitem{Cra92}
M. Cranston,
\newblock{\it A probabilistic approach to gradient estimates}, Canad. Math. Bull. {\bf 35} (1992), 46--55.

\bibitem{DL10}
G. Da Prato, A. Lunardi,
\newblock{\it On the {D}irichlet semigroup for {O}rnstein-{U}hlenbeck operators in subsets of {H}ilbert spaces}, J. Funct. Anal. {\bf 259} (2010), 2642--2672.

\bibitem{DPL14}
G. Da Prato, A. Lunardi,
\newblock{\it Sobolev regularity for a class of second order elliptic PDE's in infinite dimension},   Ann. Probab. {\bf 42} (2014), 2113--2160.

\bibitem{DL15}
G. Da Prato, A. Lunardi,
\newblock{\it Maximal {S}obolev regularity in {N}eumann problems for gradient systems in infinite dimensional domains}, Ann. Inst. Henri Poincar\'e Probab. Stat. {\bf 51} (2015), 1102--1123.

\bibitem{DaP02}
G. Da Prato, J. Zabczyk,
\newblock{Second order partial differential equations in {H}ilbert spaces}, London Mathematical Society Lecture Note Series, vol. 293, Cambridge University Press, Cambridge, 2002.


\bibitem{DS90}
J. Deuschel, D. W. Stroock,
\newblock{\it Hypercontractivity and spectral gap of symmetric diffusions with applications to the stochastic {I}sing models}, J. Funct. Anal. {\bf 92} (1990), 30--48.

\bibitem{Fer15}
S. Ferrari,
\newblock{\it Sobolev spaces with respect to a weighted Gaussian measures in infinite dimensions},
eprint arXiv:1510.08283 (2015).

\bibitem{FR82}
S. Fitzpatrick, R. R. Phelps,
\newblock{\it Differentiability of the metric projection in {H}ilbert space}, Trans. Amer. Math. Soc. {\bf 270} (1982), 483--501.

\bibitem{Fri64}
A. Friedman,
\newblock{Partial differential equations of parabolic type}, Prentice-Hall, Inc., Englewood Cliffs, N.J., 1964.

\bibitem{Hin03}
M. Hino,
\newblock{\it On {D}irichlet spaces over convex sets in infinite dimensions} in ``Finite and infinite dimensional analysis in honor of {L}eonard {G}ross ({N}ew {O}rleans, {LA}, 2001)'', Contemp. Math., vol. 317 (2003), Amer. Math. Soc., Providence, RI, 143--156.

\bibitem{Hin11}
M. Hino,
\newblock{\it Dirichlet spaces on {$H$}-convex sets in {W}iener space}, Bull. Sci. Math. {\bf 135} (2011), 667--683.

\bibitem{Hol73}
R. B. Holmes,
\newblock{\it Smoothness of certain metric projections on {H}ilbert space}, Trans. Amer. Math. Soc. {\bf 184} (1973), 87--100.

\bibitem{Lor17}
L. Lorenzi,
\newblock{Analytical methods for Kolmogorov equations}, Second Edition, Monographs and Research Notes in Mathematics, CRC Press, 2017.


\bibitem{Lun98}
A. Lunardi,
\newblock{\it Schauder theorems for linear elliptic and parabolic problems with unbounded coefficients in $\Rn$}, Studia Math. {\bf 128} (1998), 171--198.

\bibitem{LMP}
A. Lunardi, M. Miranda Jr., D.  Pallara,
\newblock{\it $BV$ functions on convex domains in Wiener spaces},
Potential Anal. {\bf 43} (2015), 23-48.

\bibitem{MR92}
Z. M. Ma, M. R\"ockner,
\newblock{Introduction to the theory of (nonsymmetric) {D}irichlet forms}, Universitext, Springer-Verlag, Berlin, 1992.

\bibitem{MvN07}
J. Maas, J. van Neerven,
\newblock{\it On analytic {O}rnstein-{U}hlenbeck semigroups in infinite dimensions}, Arch. Math. (Basel) {\bf 89} (2007), 226--236.

\bibitem{MvN11}
J. Maas, J. van Neerven,
\newblock{\it Gradient estimates and domain identification for analytic {O}rnstein-{U}hlenbeck operators} in ``Parabolic problems'', Progr. Nonlinear Differential Equations Appl., vol. 80 (2011), Birkh\"auser/Springer Basel AG, Basel, 463--477.

\bibitem{Man90}
M. Mandelkern,
\newblock{\it On the uniform continuity of {T}ietze extensions}, Arch. Math. (Basel) {\bf 55} (1990), 387--388.

\bibitem{Ouh05}
E. M. Ouhabaz,
\newblock{Analysis of heat equations on domains}, London Mathematical Society Monographs Series, vol. 31, Princeton University Press, Princeton, NJ, 2005.

\bibitem{Pri99}
E. Priola,
\newblock{\it On a class of Markov type semigroups in spaces of uniformly continuous and bounded functions}, Studia Math. {\bf 136} (1999), 271--295.

\bibitem{Rot81}
O. S. Rothaus,
\newblock{\it Logarithmic {S}obolev inequalities and the spectrum of {S}chr\"odinger operators}, J. Funct. Anal. {\bf 42} (1981), 110--120.

\bibitem{Sav14}
G. Savar\'e,
\newblock{\it Self-improvement of the {B}akry-\'Emery condition and {W}asserstein contraction of the heat flow in {${\rm RCD}(K,\infty)$} metric measure spaces}, Discrete Contin. Dyn. Syst. {\bf 34} (2014), 1641--1661.

\bibitem{Wan97}
F.-Y. Wang,
\newblock{\it On estimation of the logarithmic {S}obolev constant and gradient estimates of heat semigroups}, Probab. Theory Related Fields {\bf 108} (1997), 87--101.

\newblock
\end{thebibliography}
\end{document}